\documentclass[12pt,twoside]{amsart}
\usepackage{amssymb}
\usepackage[all]{xy}
\nonstopmode

\numberwithin{equation}{section} 

\font\tengothic=eufm10 scaled\magstep 1 \font\sevengothic=eufm7
scaled\magstep 1
\newfam\gothicfam
       \textfont\gothicfam=\tengothic
       \scriptfont\gothicfam=\sevengothic
\def\goth#1{{\fam\gothicfam #1}}

\newtheorem{theorem}{Theorem}[section]
\newtheorem{lemma}[theorem]{Lemma}
\newtheorem{proposition}[theorem]{Proposition}
\newtheorem{corollary}[theorem]{Corollary}

\theoremstyle{definition}
\newtheorem{definition}[theorem]{Definition} 
\newtheorem{remark}[theorem]{Remark}
\newtheorem{example}[theorem]{Example}

\newcommand{\codim}{\operatorname{codim}}
\newcommand{\coker}{\operatorname{coker}}
\newcommand{\aut}{\operatorname{aut}}

\newcommand{\Hom}{\operatorname{Hom}}

\newcommand{\ext}{\operatorname{ext}}
\newcommand{\Ext}{\operatorname{Ext}}

\newcommand{\depth}{\operatorname{depth}}

\newcommand{\im}{\operatorname{im}}

\newcommand{\Hi}{\operatorname{Hilb}}
\DeclareMathOperator{\GradAlg}{GradAlg}

\newcommand{\Proj}{\operatorname{Proj}}
\newcommand{\Spec}{\operatorname{Spec}}
\newcommand{\proj}[1]
{ \mathchoice
            { {\mathbb P}^{#1} }
            { {\mathbb P}^{#1} }
            { {\mathbb P}^{#1} }
            { {\mathbb P}^{#1} }
          }

\newcommand{\Tor}{\operatorname{Tor}}

\newcommand{\cA}{{\mathcal A}}
\newcommand{\cB}{{\mathcal B}}

\newcommand{\cI}{{\mathcal I}}

\newcommand{\cC}{{\mathcal C}}

\newcommand{\cF}{{\mathcal F}}
\newcommand{\cG}{{\mathcal G}}
\newcommand{\cO}{{\mathcal O}}

\newcommand {\QQ}{\mathbb{Q}}

\newcommand {\PP}{\mathbb{P}}

\newcommand {\ZZ}{\mathbb{Z}}
\newcommand {\VV}{\mathbb{V}}
\newcommand {\ra}{\longrightarrow}

\begin{document}
\title[]{Families of artinian and low dimensional determinantal rings.}

\author[Jan O.\ Kleppe]{Jan O.\ Kleppe} 

\address{Faculty of Technology, Art and Design, Oslo and Akershus University
  College of Applied Sciences, P.O. Box 4, St.\! Olavs Plass, NO-0130 Oslo,
  Norway} \email{JanOddvar.Kleppe@hioa.no}
\date{}

 \begin{abstract}
   Let ${\rm GradAlg}(H)$ be the scheme parameterizing graded quotients of $R
   =k[x_0, \dots ,x_n]$ with Hilbert function $H$ (it is a subscheme of the
   Hilbert scheme of $\PP^n$ if we restrict to quotients of positive
   dimension, see definition below). A graded quotient $A=R/I$ of codimension
   $c$ is called {\it standard determinantal} if the ideal $I$ can be
   generated by the $t \times t$ minors of a homogeneous $t \times (t+c-1)$
   matrix $(f_{ij})$. Given integers $a_0\le a_1\le ...\le a_{t+c-2}$ and
   $b_1\le ...\le b_t$, we denote by $W_s(\underline{b};\underline{a}) \subset
   {\rm GradAlg}(H)$ the stratum of determinantal rings where $f_{ij} \in R$
   are homogeneous of degrees $a_j-b_i$.   \\[2mm]
   In this paper we extend previous results on the dimension and codimension
   of $W_s(\underline{b};\underline{a})$ in ${\rm GradAlg}(H)$ to {\it
     artinian determinantal rings}, and we show that ${\rm GradAlg}(H)$ is
   generically smooth along $W_s(\underline{b};\underline{a})$ under some
   assumptions. For zero and one dimensional determinantal schemes we
   generalize earlier results on these questions. As a consequence we get that
   the {\it general} element of a component $W$ of the Hilbert scheme of
   $\PP^n$ is {\it glicci} provided $W$ contains a standard determinantal
   scheme satisfying some conditions. We also show how certain {\it ghost
     terms} disappear under deformation while other ghost terms remain and are
   present in the minimal
   resolution of a general element of $\GradAlg(H)$. \\[-2mm]

 \noindent {\bf Keywords}. Determinantal rings, Deformations, Hilbert
  schemes, Artinian rings. \\[-3mm]

  \noindent {\bf AMS Subject Classification.} 13C40, 14C05, 13D10 
  (Primary), 14M12, 13E10, 13D02, 13D07, 13A02, 13H10 (Secondary).

\end{abstract}
\maketitle


\section{Introduction} \label{intro} The main goal of this paper is to
generalize previous results on maximal families of determinantal schemes,
notable to cover maximal families of {\it artinian} determinantal
$k$-algebras. Recall that a $k$-algebra $A \simeq R/I$, $R=k[x_0, \dots
,x_n]$, of codimension $c$ is called {\it determinantal} if the ideal $I$ can
be generated by the $r \times r$ minors of a homogeneous $p \times q$ matrix
$(f_{ij})$ with $c=(p-r+1)(q-r+1)$. $A$ is called {\it standard (resp. good)
  determinantal}. if $r={\rm \min}(p,q)$ (resp. $r={\rm \min}(p,q)$ and $A$ is
a generic complete intersection of $R$).

\vskip 3 pt Let $ \GradAlg(H)$ be the ``Hilbert scheme of constant Hilbert
function'', i.e. the scheme parameterizing graded quotients $A$ of $R$ of
$\depth A \ge \min(1,\dim A)$ and with Hilbert function $H$. Given integers
$a_0\le a_1\le ...\le a_{t+c-2}$ and $b_1\le ...\le b_t$, we denote by
$W(\underline{b};\underline{a})$ (resp. $W_s(\underline{b};\underline{a}))$
the stratum in $ \GradAlg(H)$ consisting of good (resp. standard)
determinantal $k$-algebras 
where $f_{ij}$ are homogeneous polynomials of degree $a_j-b_i$. Then
$W_s(\underline{b};\underline{a}) $ is irreducible,
$W_s(\underline{b};\underline{a}) \ne \emptyset$ if $a_{i-1} > b_i$ for $1 \le
i \le t$, and the closures $ \overline {W_s(\underline{b};\underline{a})}$ and
$ \overline {W(\underline{b};\underline{a})}$ are equal if $n \ge
c$ (Proposition~\ref{WWs}). \vskip 2 pt

In this paper we focus on the following problems. \vskip 1 pt (1) Determine
when $ \overline
{W_s(\underline{b};\underline{a})}$ 
is an irreducible component of $ \GradAlg(H)$. \vskip 1 pt (2) Find the
codimension of $W_s(\underline{b};\underline{a})$ in $ \GradAlg(H)$ if its
closure is not a component. \vskip 1 pt (3) Determine when $ \GradAlg(H)$ is
generically smooth along $W_s(\underline{b};\underline{a})$.

\vskip 5 pt These questions have been considered in several papers, and in
\cite{K10} we solve all these problems completely provided $n-c >1$ and $a_0 >
b_t$. In this case the closure of $W_s(\underline{b};\underline{a})$ is a
generically smooth irreducible component of the usual Hilbert scheme $ {\rm
  Hilb}(\PP^{n})$, as well as of $ \GradAlg(H)$ (see
Theorem~\ref{Amodulethm5}), i.e. \cite[Conjecture 4.2]{KM09} holds, and $\dim
W_s(\underline{b};\underline{a})$ is determined (equal to $\lambda$ in
Theorem~\ref{Amodulethm3}, whence \cite[Conjecture 4.1]{KM09} holds for $n-c >
0$), see also \cite{FFnew} for a somewhat different approach to these
problems. So we only need to consider the case $n+1-c \in \{0,1,2\}$, and in
\cite{K09} we solve problems (1)-(3) for determinantal schemes with 
$n+1-c \in \{1,2\}$ under certain assumptions. In this paper we concentrate on
{\it artinian} determinantal $k$-algebras ($n+1-c = 0$), 
and we prove the main Theorems~\ref{Amodulethm3} and \ref{Amodulethm5} under
conditions that
allow $c$ to be $n+1$. In addition we prove a new result
(Theorem~\ref{compthmvar}, extending Theorem~\ref{Amodulethm5}) which applies
when $\dim W_s(\underline{b};\underline{a}) \ne \lambda$. This theorem implies
that the general element of an irreducible component $W$ of the Hilbert scheme
of $\PP^n$ is {\it glicci} (in the Gorenstein liaison class of a complete
intersection) provided $W$ contains a standard determinantal scheme satisfying
some conditions (Corollary~\ref{glicci}). For an introduction to
glicciness, see \cite{KMMNP}, and see \cite{EHS} and its references for
further developments. Finally, in Sec. 6, we generalize and complete
several results of \cite{K09} for families of determinantal schemes of
dimension 0 or 1, some of which occurred in
\cite{KMMNP,KM}, and we slightly extend a result of \cite{KM09}.

 We have used two different strategies to attack the problems (1) to (3).
Indeed in \cite{KMMNP,KM,K09,KM09} we successively deleted columns of the $t
\times (t+c-1)$ matrix $\cA$ associated to a determinantal scheme
$X:=\Proj(A)$ to get a nest (``flag'') of closed subschemes $X =X_c \subset
X_{c-1} \subset ... \subset X_2 \subset {\PP^{n} }$ and we proved our results
inductively by considering the smoothness of the Hilbert flag scheme of pairs
and its natural projections to Hilbert schemes. On the other hand, in
\cite{K10} we compared deformations of $A$ with deformation of the cokernel
$M$ of the map $\oplus_{j=0}^{t+c-2}R(-a_j) \to \oplus_{i=1}^tR(-b_i)$ induced
by $\cA$. This latter approach turned out to be quite successful, and it
indeed solved problem (1)-(3) for $n-c > 1$. It is this approach that we
generalize to the artinian case, only introducing an extra assumption (\!
$_0\!\Hom_A(M,M) \simeq k$) in the theorems. In fact we show that the main
results of \cite{K10} hold, whence partially solving problems (1) to (3) also
for $n=c-1$.

For $c=2$ all assumptions of the theorems are fulfilled. We even replace $R$
by any Cohen-Macaulay quotient $\overline R$ of $R$ and solve the problems
above for determinantal quotients of codimension $c=2$ of $\overline R$
(Theorem~\ref{elling}). For $c > 2$ we need to verify that certain ${\rm
  Ext^i}$-groups vanish to apply our results when $n+1-c \in \{0,1,2\}$. In
this paper we accomplish this by using Macaulay 2 (\cite{Mac}). We give many
examples, supported by Macaulay 2 computations, to illustrate the theorems in
the artinian case.

In Sec.\;5 we consider ghost terms, i.e.\;common free summands in consecutive
terms of the minimal $R$-free resolution of $A$, and we show that some of them
disappear, others remain unchanged under suitable generizations (deformations
to a more general algebra) in $\GradAlg(H)$. If $A$ is general in
$\GradAlg(H)$ and the assumptions of Theorem~\ref{compthmvar} are fulfilled
(e.g. $\dim A \ge 3$ and $a_0\ge b_t$), one may easily describe all ghost
terms in its minimal free resolution while it seems hard to get improved
results when $\dim A \le 2$ (Proposition~\ref{ghost}, Remark~\ref{ghdim2}),
although examples indicate a possible connection (Remark~\ref{ghext2}).

In the proofs we use the exactness of the Buchsbaum-Rim complex (\cite{BR,
eise}) and the 5-term exact sequence associated to the spectral sequence
$E_2^{p,q}:= \Ext_A^p(\Tor_q^R(A,M),M) \ \Rightarrow \ \Ext_R^{p+q}(M,M)$
which, in view of Fitting's lemma, plays
a surprisingly important role in comparing the deformation functors of $M$ and
$A$.

\vskip 4mm

{\bf Notation:} Throughout $\PP:=\PP^n$ is the projective $n$-space over an
algebraically closed field $k$ (except in 
Theorem~\ref{elling} where we allow $k \ne \overline k$), $R=k[x_0, x_1, \dots ,x_n]$ is a polynomial ring and
$\goth m= (x_0, \dots ,x_n)$.
If $X \subset Y$ are closed subschemes of $\PP^n$, we denote by ${\mathcal
  I}_{X/Y}$ (resp. ${\mathcal N}_{X/Y}$) the ideal (resp. normal) sheaf of $X$
in $ Y $, and we omit $/Y$ if $Y=\PP^n$. Let $ I_X=H^0_{*}({\mathcal I}_X)
\subset R$ be the saturated homogeneous ideal of $X \subset \PP^n$. 
When {\it we write} $X=\Proj(A)$ {\it we always let} $A:=R/I_X$ and
$K_A=\Ext^c_R (A,R)(-n-1)$ be the canonical module of $A$ or $X$ where
$c=\codim_{\PP} X:=n-\dim X$. Note that by the codimension, ${\rm codim}_Y X$,
of an irreducible $X$ in a not necessarily equidimensional scheme $Y$ we mean
$\dim {\mathcal O}_{Y,x} -\dim X$, where $x$ is a general $k$-point of $X$. We
let ${}_v\! \Ext(F,G)$ be the degree-$v$ part of the graded module
$\Ext(F,G)$. Moreover we denote by $\hom(\cF,\cG)=\dim_k\Hom(\cF,\cG)$
the dimension of the group of morphisms between coherent $\cO_X$-modules and
we use small letters for the $k$-vector space dimension of similar groups.

\section{Preliminaries}

\subsection{Hilbert schemes and Hilbert function strata}

We denote the Hilbert scheme with the Hilbert polynomial $p \in \QQ[s]$ by
$\Hi ^p(\PP^n)$, cf.\! \cite{G} for existence and \cite{L} for the local
theory. 
Similarly $\GradAlg^H(R)$, or $\Hi ^H(\PP^n)$ when $\dim A >0$, is the
representing object of the functor that parameterizes flat families of graded
quotients $A$ of $R$ of $\depth_{\goth m} A \ge \min(1,\dim A)$ and with
Hilbert function $H$, $H(i)= \dim A_i$ (\cite{K98,K04,HS}). We allow calling
it ``the postulation Hilbert scheme'' (\cite[\S 1.1]{K07}) even though it may
be different from the parameter space studied by Gotzmann and 
Iarrobino 
(\cite{Go,IK}) who study the ``same'' scheme with the reduced scheme
structure.
We let $(A)$, or $(X)$ where $X= \Proj(A)$, denote the point of $\GradAlg(H):=
\GradAlg^H(R)$ that corresponds to $A$. Note that if $\depth_{\goth m}A \geq
1$ and $\ _0\!\Hom_R (I_X,H^{1}_{\goth m}(A)) = 0$, then
\begin{equation}  \label{Grad}
  \GradAlg(H) \simeq \Hi ^p(\PP^n) \ \ \ {\rm at} \ \ \ (X)\, ,
\end{equation} 
and hence we have an isomorphism $ _0\!\Hom (I_{X},A) \simeq \ H^0({\mathcal
  N}_X)$ of their tangent spaces (cf. \cite{elli} for the case $\depth_{\goth
  m}A \geq 2$, and (9) of \cite{K04} for the general case). If $X$ is
generically a complete intersection, then $ _0\!\Ext^1_A(I_{X}/I_{X}^2,A)$ is
an obstruction space of $ \GradAlg(H)$ at $(A)$, and hence of $ \Hi ^p(\PP^n)$
at $(X)$ if \eqref{Grad} holds. Indeed for any quotient $A=R/I$, we may define
${\rm H}_2(R,A,A):= \ker(S_2(I) \to I^2)$ where $S_2(I)$ is the $2^{\rm nd}$
symmetric power of $I$, cf.\! \cite{AND}, p.106. Then $ \GradAlg(H)$ is smooth
at $(A)$ if
\begin{equation}  \label{GradGenCI}
 _0\!\Ext^1_A(I/I^2,A)=0 \quad {\rm and} \quad
 _0\!\Hom_A({\rm H}_2(R,A,A),A)=0 \ , 
\end{equation} 
and observe that the latter group vanishes if
A is generically a complete intersection of $R$ \cite[\S 1.1]{K04}.
By definition $X$ (resp. $A$) is called {\it unobstructed} if $\Hi ^p(\PP^n)$
(resp. $\GradAlg(H)$) is smooth at $(X)$ (resp.
$(A)$). 

We say that $X$ is {\it general} in some irreducible subset $ W \subset
\Hi^p(\PP^n)$ if $(X)$ belongs to a sufficiently small open subset $U$ of $W$
such that any $(X)$ in $U$ has all the openness properties that we want to
require. We define "$A$ general in  $\GradAlg(H)$" similarly.

\subsection{Determinantal rings and schemes}

We mainly maintain the notions and notations from \cite{K09,K10}, but we need
to extend some results to artinian determinantal $k$-algebras. For a more
general background of determinantal rings and schemes, see \cite{b-v,eise, BH,
  rmmr}. 

Indeed let
\begin{equation}\label{gradedmorfismo} \varphi:F=\bigoplus
  _{i=1}^tR(b_i)\longrightarrow G:=\bigoplus_{j=0}^{t+c-2}R(a_j)
\end{equation}
be a graded morphism of free $R$-modules, let
$\cA=(f_{ij})_{i=1,...t}^{j=0,...,t+c-2}$, $\deg f_{ij}=a_j-b_{i}$, be a
matrix which represents the dual $\varphi^*:=\Hom_R(\varphi,R)$ and let
$I(\cA):=I_t(\cA)$ be the ideal of $R$ generated by the maximal minors of
$\cA$. We {\it always suppose} $t\ge 2$ (unless explicitly allowing $t=1$) and
\begin{equation} \label{ba}
 c\ge 2,\ \ b_1 \le ... \le b_t \ \ \  {\rm and} \ \ \ a_0 \le
 a_1\le ... \le a_{t+c-2}.
\end{equation}
A codimension $c$ quotient $A=R/I$ (resp. subscheme $X\subset \PP^{n}$ if $n
\ge c$) is called \emph{standard
  determinantal} 
if $I:=I(\cA)$ (resp. $I_X:=I(\cA)$) for some homogeneous $t\times (t+c-1)$
matrix $\cA$ as
above. 
They are \emph{good determinantal} if additionally, the ideal $I_{t-1}(\cA)$
of submaximal minors has at least codimension $c+1$ in $R$, i.e.\;$R/I(\cA)$
is a generic complete intersection of $R$. If $t=1$, standard, as well as good
determinantal are complete intersections of $ R$.

Given integers $b_i$ and $a_j$ satisfying \eqref{ba} we let
$W_s(\underline{b};\underline{a};R)$ (resp.
$W(\underline{b};\underline{a};R)$) be the stratum in
$\GradAlg(H):=\GradAlg^H(R) $ consisting of standard (resp. good)
determinantal quotients, cf.\;\eqref{Grad}, and we usually omit $R$ in this
notation. Note that we do not require $\cA$ to be minimal (i.e. $f_{ij}=0$
when $b_{i}=a_{j}$) for $R/I_{t}(\cA)$ to belong to
$W_s(\underline{b};\underline{a}):=W_s(\underline{b};\underline{a};R)$. In
examples, however, we mainly consider determinantal rings with positive degree
matrix (i.e.\! for every $i,j$, $b_{i}<a_{j}$, and recall that $a_{j}-b_{i}=1$
is called the linear case). 
\begin{proposition} \label{WWs}  The closure of \!
  $W_s(\underline{b};\underline{a})$ in  $\GradAlg(H)$ is irreducible, and
  \begin{equation*} \label{WWs} \ W_s(\underline{b};\underline{a}) \ne
    \emptyset \ \ \Leftrightarrow \ \ a_{i-1} \ge b_i \ \ {\rm for \ all} \ i
    \ {\rm and } \ a_{i-1} > b_i \ \ {\rm for \ some } \ i \ .
\end{equation*}
Moreover if \ $n-c \ge 0$, then $\overline {W_s(\underline{b};\underline{a})}
= \overline {W(\underline{b};\underline{a})}$. And if \ $n-c=-1$, then every
$A=R/I$ of ${W(\underline{b};\underline{a})}$ satisfies $I_{t-1}(\cA)=R$, i.e.
is a complete intersection (c.i.) of $R$.
\end{proposition}
\begin{proof} For $W_s(\underline{b};\underline{a})$ to be non-empty we refer
  to $(2.2)$ of \cite{KM09} whose arguments carry over to the artinian case.
  The text after \cite[(2.2)]{KM09} shows also that
  $ \overline {W_s(\underline{b};\underline{a})} = \overline
  {W(\underline{b};\underline{a})}$
  and that this locus is irreducible in the non-artinian case ($n \ge c)$. To
  see that $ W_s(\underline{b};\underline{a})$ is irreducible also when
  $ W_s(\underline{b};\underline{a})$ parametrizes artinian $k$-algebras we
  consider the affine scheme $\VV =\Hom_{R}(G^*,F^*)$ whose rational points
  correspond to $t \times (t+c-1)$ matrices. Since the vanishing of
  $\Ext_R^i(R/I_{t}(\cA),R)$ is an open property, the subset $U$ of $\VV$ of
  matrices such that $I_{t}(\cA)$ has maximal codimension in $R$ is open and
  irreducible. Then, since there is an obvious morphism from $U$ onto
  $W_s(\underline{b};\underline{a})$, it follows that
  $\overline{ W_s(\underline{b};\underline{a})}$ is irreducible.

  Finally, for $A$ artinian and good determinantal, $I_{t-1}(\cA)$ is not
  contained in ${\goth m}$ because this would imply $\dim R/I_{t-1}(\cA) =
  \dim R/I_{t}(\cA)=0$. Thus $I_{t-1}(\cA)=R$, and it follows that $A$ is a
  c.i. of $R$ by Corollary~\ref{corghost}.
\end{proof}

\begin{remark} \label{earlygood} In earlier papers on determinantal rings we
  often assumed that $A$ was good determinantal. By \cite[Lem.\! 3.2]{KM} good
  determinantal rings satisfy $A \simeq \Hom_A(M,M)$ where
  $M:=\coker \varphi^*$. This and the fact that they are generic complete
  intersections of $R$, was often used in the proofs of the deformation
  results in \cite{KMMNP,KM,K09,K10}. To generalize earlier results to the
  artinian case, we see from Proposition~\ref{WWs} that we can not assume $A$
  good determinantal any more. Also $A \simeq \Hom_A(M,M)$ has to be dropped
  because this seems to never hold for $A$ artinian. It turns, however, out
  that only assuming $_0\! \Hom_A(M,M) \simeq k$ and $A$ standard
  determinantal often allows us to generalize the proofs.
\end{remark}

In the following let $A=R/I_{t}(\cA)$ be standard determinantal, i.e. $(A) \in
W_s(\underline{b};\underline{a})$ and let $X=\Proj(A)$ if $\dim A > 0$. Then
$R$-free resolutions of $A$ and $M:=M_{\cA}:=\coker \varphi^*$ are given by
the {\em Eagon-Northcott} and {\em Buchsbaum-Rim complexes} respectively, see
\cite{BR, e-n, eise}. These resolutions are minimal if $\cA$ is minimal.
By e.g. \cite[(2.3), (2.4)]{K10} the resolutions are\\[-2mm]
\begin{equation}\label{ENres}
 \begin{split}
   0 \ra \wedge^{t+c-1}G^* \otimes S_{c-1}(F)\otimes \wedge^tF\ra \cdots \to
   \wedge^{t+i} G ^*\otimes S _{i}(F)\otimes \wedge ^tF \\ \to \ldots \ra
   \wedge^{t}G^* \otimes S_{0}(F)\otimes \wedge^tF\ra R \ra A \ra 0 \ ,\quad
   {\rm and } \
 \end{split}
\end{equation}
\begin{equation}\label{BR}
  \begin{split}
0 \to \wedge^{t+c-1}G^* \otimes S_{c-2}(F)\otimes \wedge^t F \to
\cdots \to \wedge^{t+i+1} G^* \otimes S _{i}(F )\otimes \wedge ^t F \\
\to \ldots \to \wedge^{t+1}G^* \otimes S_{0}(F)\otimes \wedge^t F \to G^*
\stackrel {\varphi^*} { \longrightarrow} F^* \to M \to 0 \ .
  \end{split}
\end{equation}
The matrix ${\cA}_i$  obtained by deleting the last
$c-i$ columns defines a morphism
\begin{equation}\label{gradedmorfismo} \varphi
_i:F=\bigoplus _{i=1}^tR(b_i)\longrightarrow G_i:=\bigoplus
_{j=0}^{t+i-2}R(a_j) \ .
\end{equation}
and we suppose that {\it all entries of the deleted columns belong to}
$\goth m$. Then the $k$-algebra $D_i\cong R/I_{D_i}$ given by the maximal
minors of ${\cA}_i$ is standard determinantal by \cite{Bru}
and 
$\varphi_i$ is injective. Letting $B_i= \coker \varphi_i$ and
$M_i = \coker \varphi _i^*$, 
we have an exact sequence
\begin{equation}\label{defMi}
  0 \to B_i^* \to G_i^*\stackrel {\varphi_i^*}{ \longrightarrow} F^*
  \longrightarrow   M_i\cong  \Ext ^1_R(B_i,R)\longrightarrow 0 \ . 
\end{equation}
Then $M_i$ is a $D_i$-module,
$R\twoheadrightarrow D_i \twoheadrightarrow D_{i+1} 
...\twoheadrightarrow D_c=A$, and there is an exact sequence
\begin{equation}\label{Mi}
0\longrightarrow D_i \longrightarrow M_i(a_{t+i-1})
\longrightarrow M_{i+1}(a_{t+i-1}) \longrightarrow 0
\end{equation}
in which $D_i \longrightarrow M_i(a_{t+i-1})$ is a regular section that
defines $D_{i+1}$ (also when $i+1=c$ and $D_c$ is artinian
\cite[Lem.\,3.6]{KMNP}). Hence the sequence
\begin{equation}\label{Di}
  0 \longrightarrow
  I_{D_{i+1}/D_i} \simeq \Hom_{D_i}(M_i(a_{t+i-1}),D_i)\longrightarrow D_i
  \longrightarrow D_{i+1}\longrightarrow 0 \ .
\end{equation}
is exact. By \eqref{BR} $M_i$ is a maximal Cohen-Macaulay $D_i$-module and so
is $ I_{D_{i+1}/D_i} $ by (\ref{Di}). Moreover by e.g.\;\eqref{BR},
$K_{D_i}(n+1)\cong S_{i-1}M_i(\ell _i)$ where
$\ell_i:=\sum_{j=0}^{t+i-2}a_j-\sum_{q=1}^tb_q$.
\begin{lemma} \label{key} With the above notation (and $M=M_c$), there are
  exact sequences
$$0\rightarrow \Hom_R(M_i,M)\rightarrow F\otimes_R M \rightarrow
G_i\otimes _R M\rightarrow B_i\otimes _R M\rightarrow 0\, , \ {\rm and}$$ 
$$0\rightarrow \Hom_R(B_i,F)\rightarrow \Hom_R(B_i,G_i)\rightarrow
\Hom_R(B_i,B_i)\rightarrow \Hom_R(M_i,M_i)\rightarrow 0 \ .$$
\end{lemma}
\begin{proof} (cf. \cite[Lem.\! 3.1 and 3.10]{KM}). We apply  $\Hom_R(B_i,-)$ to
\begin{equation}\label{Bi}
  0 \longrightarrow F\longrightarrow G_i\longrightarrow B_i
  \longrightarrow 0 \end{equation} and using $M_i \simeq \Ext^1_R(B_i,R)$ we
deduce the exact sequence    
$$ 0\rightarrow
\Hom(B_i,F)\rightarrow \Hom_R(B_i,G_i)\rightarrow \Hom_R(B_i,B_i)\rightarrow
F\otimes _R M_i \rightarrow G_i \otimes _R M_i\, .$$ Hence we get the lemma by
applying $\Hom(-,M_i)$ to \eqref{defMi} and $(-)\otimes _R M$ to (\ref{Bi}).
\end{proof}
\begin{proposition} \label{autdim}
  \label{2Aut(B)} Set $K_i:=\ _0\! \hom(B_{i-1},R(a_{t+i-2}))$ and
  $\ell_i:=\sum_{j=0}^{t+i-2}a_j-\sum_{q=1}^tb_q$. Letting $h_{i-3}:=
  2a_{t+i-2}-\ell_i +n$ for $3\le i \le c$ and $ \binom{a}{n}= 0$ if $a < n$
  we have
  \[ K_{i+3}= \sum _{r+s=i \atop r , s \ge 0} \sum _{0\le i_1< ...< i_{r}\le
    t+i \atop 1\le j_1\le...\le j_s \le t } (-1)^{i-r}
  \binom{h_i+a_{i_1}+\cdots +a_{i_r}+b_{j_1}+\cdots +b_{j_s} }{n} \ {\rm for}
  \ 0 \le i \le c-3 \, , \]
  e.g. $K_3=\binom{h_0}{n}$ and
  $K_4= \sum_{j=0}^{t+1} \binom{h_1+a_j}{n}- \sum_{i=1}^{t}
  \binom{h_1+b_i}{n}$.
  Moreover if \ $_0\! \hom(M,M)=1$ and the entries of the last $c-2$ columns
  belong to $\goth m$ (e.g. $\cA$ is minimal), then
$$\aut(B_c):=\ _0\! \hom(B_c,B_c) = 1+K_3+K_4+\cdots+K_c\, .$$
\end{proposition}
\begin{proof}
  Using (\ref{BR}) we get a resolution of $M_{i-1}$ which due to 
  \eqref{defMi} leads to the following resolution: 
$$ 0\rightarrow
E_{i-3} \rightarrow \cdots \rightarrow E_j \rightarrow \cdots \rightarrow E_0
\to B_{i-1}^* \to 0 \ $$ where $E_j:= \wedge^{t+j+1} G^* \otimes S _{j}(F
)\otimes \wedge ^t F$. Taking dimensions of the degree $a_{t+i-2}$ part of the
free modules in this resolution and using $K_i = \dim B_{i-1}^*(a_{t+i-2})_0$,
we get that $K_i$ coincides with the sum of binomials appearing in the
proposition (see \cite[ p.2882]{KM} for details).

  Now dualizing the exact sequence $0\rightarrow R(a_{t+c-2})\rightarrow
  B_c\rightarrow B_{c-1}\rightarrow 0$, we get $$ 0\rightarrow
  \Hom(B_{c-1},R)\rightarrow \Hom_R(B_c,R)\rightarrow R(-a_{t+c-2})\rightarrow
  M_{c-1} \rightarrow M_c\rightarrow 0 \, $$ by (\ref{defMi}). Combining with
  (\ref{Mi}) we get the exact sequence
  \begin{equation*} 0\rightarrow \Hom(B_{c-1},R)\rightarrow
    \Hom_R(B_c,R)\rightarrow I_{D_{c-1}}(-a_{t+c-2})\rightarrow
    0 \end{equation*}
which implies the vanishing of the {\it lower down-arrows} in the
following commutative diagram
\[
\begin{array}{cccccc}
 _0\!\Hom(B_{c-1},F) &  \longrightarrow &   _0\!\Hom(B_{c-1},G_c)
\\ \downarrow  & &
\downarrow  &     \\ _0\!\Hom(B_{c},F) &  \longrightarrow &
 _0\!\Hom(B_{c},G_c)
\\ \downarrow  & & \downarrow   \\
\ _0\!\Hom(R(a_{t+c-2}),F) &  \longrightarrow &
_0\!\Hom(R(a_{t+c-2}),G_c)  \, . 
\end{array}
\]
It follows that the upper down-arrows are bijections. Since $ \Hom(B_{c-1},-)$
is exact on $0 \to R(a_{t+c-2}) \to G_{c} \to G_{c-1} \to 0$ (because
$ \Ext^1(B_{c-1},-) \cong M_{c-1} \otimes (-)$) we get
\begin{equation} \label{Kcind} _0\!\hom(B_c,G_c)-\ _0\!\hom(B_c,F)= \
  _0\!\hom(B_{c-1},G_{c-1})- \ _0\!\hom(B_{c-1},F)+K_c \, .
 \end{equation}
 By Lemma \ref{key} and the assumption $_0\! \hom(M_c,M_c)=1$, we
 have $$\aut(B_c)=1+\ _0\!\hom(B_c,G_c)-\ _0\!\hom(B_c,F)\, ,$$ and since
 there is a corresponding expression for $\aut(B_{c-1})$
we get 
\begin{equation} \label{autind} \aut(B_c)=1+K_c+\aut(B_{c-1})-\ _0\!
  \hom(M_{c-1},M_{c-1}) \ \end{equation} by \eqref{Kcind}. $D_{c-1}$ is,
however, good determinantal because
$I_{t}(\cA_{c-1}) \subset I_{t}(\cA_c) \subset I_{t-1}(\cA_{c-1})$ and
$\codim_{D_{c-1}} D_c=1$. Indeed all $D_j$, $2 \le j \le c-1$ are good
determinantal for the same reason and we get $_0\! \hom(M_j,M_j)=1$ for
$2 \le j \le c-1$ by Remark \ref{earlygood}. This simplifies \eqref{autind}
which holds also if we replace $c$ by $j+1$ in \eqref{autind}. Then we
conclude the proof by induction using that $\aut(B_2)=$
$ _0\!\hom (I_{D_2},I_{D_2})=1$.
\end{proof}

Finally recall that if $J=I_{t-1}(\cA)$ and $\dim A > 0$ then $X = \Proj(A)$
is a local complete intersection (l.c.i.) in $X-V(J)$. In the following we
{\it always take} $Z \supset V(J)$ and $U=X-Z$, i.e. so that $X \hookrightarrow
\PP^{n}$ is an l.c.i. in $U$. Since the $1^{\rm st}$ Fitting ideal of $M$ is
equal to $J$, we get that $\tilde{M}$ and $\widetilde {I/I^2}$ are locally
free on $X-V(J)$, cf. \cite[ Lem.\! 1.8]{ulr} and \cite[Lem.\! 1.4.8]{BH}.
Note that $\depth_J A \ge \codim_{X}Sing(X)$ which we use in the following.

\begin{remark} \label{dep} Let $X=X_c$,  $Y=X_{c-1}$  and let
  $\alpha$ be a positive integer. If $X$ is general in
  $W(\underline{b};\underline{a}) \ne \emptyset$ and $a_{i-\min (\alpha
    ,t)}-b_i\ge 0$ for $\min (\alpha ,t)\le i\le t$, then
\begin{equation}\label{a-b} \codim_{X_j}Sing(X_j)\ge \min\{2\alpha
  -1, j+2\}  \ \ {\rm for \  \ }  2 \le j \le c \ , \end{equation}
cf.\,the arguments in \cite[Rem.\! 2.7]{KM} which use \cite{chang,W}. 
In particular letting $\alpha = 3$ we get that $X \hookrightarrow \PP^{n}$
(resp. $Y \hookrightarrow \PP^{n}$ and $X \hookrightarrow Y$)
are l.c.i.s outside a subset $Z \subset X$ of codimension at least $
\min(5,c+2)$ (resp. $ \min(4,c)$) in $X$. Indeed we may take
$Z=V(I_{t-1}(\cA_{c-1}))$ for the statement involving $X$ and $Y$ because
\eqref{Di} implies that $\cI_{X/Y}$ is locally free on $Y-Z$,
noting that $Z \subset
V(I_{t}(\cA))=X$. Moreover if $a_{i-2}-b_i\ge 0$ for $2 \le i\le t$, then $X
\hookrightarrow \PP^{n}$ (resp. $Y \hookrightarrow \PP^{n}$ and $X
\hookrightarrow Y$) are l.c.i.s outside a subset $Z \subset X$ of codimension
at least $ \min(3,c+2)$ (resp. $ \min(2,c)$). Notice that we interpret $I(Z)$
as $\goth m$ if $Z= \emptyset$.
\end{remark}

 \subsection{The dimension of the determinantal locus}

In \cite{K10} we proved that the dimension of a non-empty
$W(\underline{b};\underline{a})$ in the case $a_{i-2}-b_i \ge 0$ for $i \ge 2$
and $n-c \ge 1$  is given by
\begin{equation} \label{dimW} \dim {W_s(\underline{b};\underline{a})} = \dim
  W(\underline{b};\underline{a}) = \lambda_c + K_3+K_4+...+K_c \,,
 \end{equation}
 where $K_i$ is defined in Proposition \ref{autdim} and $\lambda_c$ is
defined by {\small
\begin{equation} \label{lamda} \lambda_c:= \sum_{i,j}
    \binom{a_j-b_i+n}{n} + \sum_{i,j} \binom{b_i-a_j+n}{n} - \sum _{i,j}
    \binom{a_i-a_j+n}{n}- \sum _{i,j} \binom{b_i-b_j+n}{n} + 1 
  \end{equation}
}where the indices belonging to $a_j$ (resp. $b_i$) range over $0\le j \le
t+c-2$ (resp. $1\le i \le t$).
\begin{remark} \label{dimWcA} Note that we assume \eqref{ba} and 
  delete columns from the right-hand side for \eqref{dimW} to hold. From the
  proof of Proposition \ref{autdim} we see that we also need all algebras
  $D_i$ in $R\twoheadrightarrow D_2 \twoheadrightarrow
  D_{3} 
  ...\twoheadrightarrow D_c=A$
  to be standard determinantal. If $D_i$ is standard determinantal and we
  delete a column whose entries belong to $\goth m$, then $D_{i-1}$ is
  standard determinantal by \cite{Bru}, see Proposition \ref{ghost} for
  deleting columns containing units. In particular \eqref{dimW} holds if the
  entries of the last $c-2$ columns of $\cA$ belong to $\goth m$ (e.g. if
  $\cA$ is minimal).
  \end{remark}
  Using, however, the Hilbert function, $H_M(-)$, of $M$ and assuming \
  $_0\! \hom(M,M)=1$, we have by \cite[Rem.\! 3.9]{K10}, for $\cA$ possibly
  non-minimal, that
  \begin{equation} \label{dimWH} \dim {W_s(\underline{b};\underline{a})} =
    \dim W(\underline{b};\underline{a})= \sum_{j=0}^{t+c-2} H_M(a_j)-
    \sum_{i=1}^{t} H_M(b_i) +1 \, . \end{equation} For zero-dimensional
  determinantal schemes ($n-c = 0$) we have the following.

\begin{remark} \label{dim0new}(i) Assume $a_{0} > b_{t}$. Then 
\eqref {dimW} holds provided $3 \le c \le 5$
  (resp. $c>5$) and$ \ a_{t+c-2} > a_{t-2} $ (resp. $ \ a_{t+3} > a_{t-2}$) by
  \cite[Thm.\! 3.2]{KM09}, see Sec.\! 6 for a generalization.

  (ii) If $\cA$ is a linear $2 \times (c+1)$ matrix, we showed in \cite[Ex.\!
  3.3]{K09} that $\dim W(\underline{b};\underline{a})$ is strictly smaller
  that the right-hand side of \eqref{dimW} for every $c >2$. To our knowledge
  this is the only known case where the equality of \eqref{dimW} fails to hold
  when $n=c$.
\end{remark}


\subsection{The determinantal locus as components of Hilbert schemes} 

It was shown in \cite{K10} that $ \overline {W(\underline{b};\underline{a})}$,
for $W(\underline{b};\underline{a}) \ne \emptyset $, is an irreducible
component of $\Hi^p(\PP^n)$ provided $n-c \ge2$ and $a_{i-\min (3,t)}-b_i\ge
0$ for $\min (3,t)\le i\le t$. If $n-c \le 1$ then $ \overline
{W(\underline{b};\underline{a})}$ may fail to be an irreducible
component. 
Indeed $ \overline {W({0,0};{1,1,...,1,2})}$ is not an irreducible component
of $\GradAlg(H)$ for all $c \ge 3$ both when $n=c$ and $n=c+1$, see
\cite[Ex.\! 4.1]{K09}. In both cases the $h$-vector of a general $A$ is
$(1,c,c)$, e.g. the Hilbert function is given by $(\dim A_i)_{i=0}^{\infty}=(1,c+1,2c+1,2c+1,...)$ if $n=c$.

The reason why $ \overline {W(\underline{b};\underline{a})}$ is not a
component is that there exist determinantal rings allowing ``deformations that
do not come from deforming the matrix $\cA$''. Let us recall this notion. We
briefly say ``$T$ a local ring'' (resp.\;``$T$ artinian'') for a local
$k$-algebra $(T,{\mathfrak m}_T)$ essentially of finite type over
$k=T/{\mathfrak m}_T$ (resp. of finite type over $k=T/{\mathfrak m}_T$ such
that ${\mathfrak m}_T^r=0$ for some $r \in \mathbb
Z$). 
The local deformation functor ${\rm Def}_{A/R}$ is defined for each artinian
$T$ as the set of graded ($T$-flat) deformations $A_T$ of $A$ to $T$ (i.e.
$A_T \otimes_T k =A$). Moreover ``$T \rightarrow S$ is called a small artinian
surjection'' if there is a morphism
$(T,{\mathfrak m}_T) \rightarrow (S, {\mathfrak m}_S)$ of local artinian
$k$-algebras whose kernel ${\mathfrak a}$ satisfies
${\mathfrak a} \cdot {\mathfrak m}_T=0$.

If $T$ is a local ring, we let $\cA_T=(f_{ij,T})$ be a matrix of
homogeneous polynomials of the graded algebra $R_T:=R
\otimes_k T$, satisfying $f_{ij,T} \otimes_T k=f_{ij}$ and $\deg f_{ij,T}
=a_j-b_i$ 
for all $i,j$. Here all elements of $T$ are considered to be of degree
zero. Once having $\cA_T$, it induces a morphism
\begin{equation} \label{al} \varphi_T: F_T:=\oplus _{i=1}^tR_T(b_i)\rightarrow
  G_T:=\oplus_{j=0}^{t+c-2}R_T(a_j) \ .
\end{equation}
\begin{lemma} \label{defalpha} If $A=R/I_t(\cA)$ is standard determinantal,
  then $A_T:=R_T/I_t(\cA_T)$ and $M_T:= \coker \varphi_T^*$ are graded
  deformations of $A$ and $M$ respectively for every choice of $\cA_T$, $T$ a
  local ring. 
Moreover every graded deformation of $M$ is of the form $M_T$ for
  some $\cA_T$.
\end{lemma}
\begin{proof} See \cite[Lem.\;4.2]{K09} for a proof. Indeed
  $A_T=R_T/I_t(\cA_T)$ and $M_T$ are $T$-flat because all maps in their
  Eagon-Northcott and Buchsbaum-Rim complexes are defined in terms of $\cA_T$,
  and these complexes become exact after tensoring by $k$ over $T$. Note also
  that the final statement in the lemma follows immediately from $M= \coker
  \varphi^*$.
\end{proof}
The final statement may be formulated as ``every graded deformation of $M$ to
a local ring $T$ comes from deforming $\cA$''. This may not hold for $A$.
\begin{definition} \label{everydef} Let $A=R/I_t(\cA)$. We say ``every
  deformation of $A$ (or $X$ if $\dim A > 1$, see
  \eqref{Grad}) 
  comes from deforming $\cA$'' if for every local ring $T$ and every graded
  deformation $R_T \to A_T$ of $R \to A$ to $T$, then $A_T$ is of the form
  $A_T=R_T/I_t(\cA_T)$ for some $\cA_T$ as
  above. 
\end{definition}
\begin{lemma} \label{unobst} Let $A=R/I_t(\cA)$ be a standard determinantal
  ring, $(A) \in W_s(\underline{b};\underline{a})$. If every deformation of $A$
  comes from deforming $\cA$, then $A$ is unobstructed (i.e. ${\rm
    Def}_{A/R}$ is smooth). Moreover $\overline{
    W_s(\underline{b};\underline{a})}$ is an irreducible component of
  $\GradAlg(H).$
\end{lemma}
\begin{proof} See \cite[Lem.\;4.4]{K09}, only replacing $ \Hi (\PP^{n})$ by
  $\GradAlg(H)$ in its proof.
\end{proof}
\begin{remark} By these lemmas we get $T$-flat determinantal schemes by just
  parameterizing the polynomials of
  $\cA$ over a local ring $T$, see Rem.\! 4.5 of \cite{K09} and Laksov's
  papers \cite{dan2,dan} for somewhat similar results for more general
  determinantal schemes. 
\end{remark} 

\section{deformations of modules and determinantal rings}

The main goal of this section is to generalize to $A$ artinian the close
relationship between the local deformation functor, ${\rm Def}_{M/R}$, of the
graded $R$-module $M=\coker \varphi^*$ and the corresponding local functor,
${\rm Def}_{A/R}$, of graded deformations of the standard determinantal ring
$A=R/I_t(\cA)$. 
Note that $I:=I_t(\cA)={\rm ann}(M)$ by \cite{eis}. In \cite{K10} these
functors were shown to be isomorphic (resp. the first a natural subfunctor of
the latter) provided $\dim X \ge 2$ (resp. $\dim X = 1$) and $X=\Proj(A)$
general and good determinantal. 
The comparison between 
these deformation functors relied on understanding the spectral sequence
$E_2^{p,q}:= \Ext_A^p(\Tor_q^R(A,M),M) \ \Rightarrow \ \Ext_R^{p+q}(M,M)$ and
its induced 5-term exact sequence {\small
  \begin{equation} \label{specseq} 0 \to \Ext_A^1(M,M) \to \Ext_R^1(M,M)
    \stackrel {\delta}{ \longrightarrow} E_2^{0,1} \to \Ext_A^2(M,M) \to
    \Ext_R^2(M,M)\to \ . 
\end{equation}
}Indeed $ \Tor_1^R(A,M) \simeq I \otimes_R M$ implies
$E_2^{0,1} \simeq \Hom_R(I,\Hom_R(M,M))$ in general, and since
$\Hom_A(M,M) \simeq A$ provided $A$ is {\it good} determinantal (i.e.
$\depth_{I_{t-1}(\cA)A}A \ge 1$) by
\cite[Lem.\! 3.2]{KM},
it follows in this case that the morphism $\delta$ of
\eqref{specseq} induces a natural map
\begin{equation} \label{specseq1} 
   _0\! \Ext_R^1(M,M) \longrightarrow \  (E_2^{0,1})_0 \simeq \ _0\!
   \Hom_R(I,A) 
\end{equation}
between the tangent spaces of ${\rm Def}_{M/R}$ and ${\rm Def}_{A/R}$. For $A$
standard determinantal, there is still a natural map \ $e_M(T):{\rm
  Def}_{M/R}(T) \to {\rm Def}_{A/R}(T)$, $T$ artinian, obtained by taking a
matrix $\cA_{T}$ whose corresponding morphism has $M_{{T}}$ as cokernel and
letting $A_T:=R_T/I_t(\cA_{T})$ (see Lemma~\ref{defalpha}). Since matrices
inducing the same $M_T$ define the same ideal of maximal minors
(Fitting's lemma, \cite[Cor.\! 20.4]{eise}), this morphism is well-defined.

\begin{remark} The main difficulty in generalizing the comparison above to an
  artinian $A$ is that $A \to \Hom_A(M,M)$ is no longer an isomorphism (for
  $A$ not a c.i.) which implies that the tangent map of \
  ${\rm Def}_{M/R} \to {\rm Def}_{A/R}$ may not be the map in
  \eqref{specseq1}, i.e. $(E_2^{0,1})_0 \simeq \ _0\! \Hom_R(I,A)$ may fail.
  The morphism \ ${\rm Def}_{M/R} \to {\rm Def}_{A/R}$ is, however,
  well-defined also in the artinian case, and since $I_t(\cA)={\rm ann}(M)$
  implies that $A \to \Hom_A(M,M)$ is injective and the functors are
  pro-representable provided $_0\!\Hom_A(M,M) \simeq k$, we shall see that we
  are able to generalize the comparison.
\end{remark}
%
\begin{definition} \label{subf} Let $A=R/I_t(\cA)$ be a standard determinantal
  ring and let $\underline {\ell}$ be the category of artinian $k$-algebras
  (cf. the text before \eqref{al}). Then the local deformation functor ${\rm
    Def}_{A \in W_s(\underline{b};\underline{a})}$, defined on $\underline
  {\ell}$, is the subfunctor of ${\rm Def}_{A/R}$ given by: \vspace{0.2cm}
\begin{equation*}
{\rm Def}_{A \in
  W_s(\underline{b};\underline{a})}(T) = \left\{ A_T \in {\rm
    Def}_{A/R}(T) \arrowvert A_T=R_T/I_t(\cA_T) {\rm \ for \ some \
    matrix \     \cA_T \ lifting \ \cA \ to \ } T \right\}.
\end{equation*}
\vspace{-0.3cm}
\end{definition}

With this definition it is obvious that the natural map $e_M:{\rm Def}_{M/R}
\to {\rm Def}_{A/R}$ defined above for any $T \in {\rm ob}(\underline {\ell})$
factors via $ {\rm Def}_{A \in W_s(\underline{b};\underline{a})}
\hookrightarrow {\rm Def}_{A/R}$\! . Moreover we have
\begin{lemma} \label{specFitt} If $D:=k[\epsilon]/(\epsilon^2)$ 
  then the degree zero part, $\delta_0(M)$, of the connecting homomorphism
  $\delta$ of \eqref{specseq} factors through $e_M(D): {\rm Def}_{M/R}(D) \to
  {\rm Def}_{A/R}(D)$, i.e. $\delta_0(M)$ is the composition
  \begin{equation*} \label{specseqlem} _0\! \Ext_R^1(M,M) \stackrel {e_M(D)}{
      \longrightarrow} \ _0\! \Hom_R(I,A) \stackrel {i_{A,M}}{ \longrightarrow}
    \ _0\! \Hom_R(I,\Hom_A(M,M)) \simeq (E_2^{0,1})_0
\end{equation*}
where the map in the middle is induced by $A \hookrightarrow \ \Hom_A(M,M)$,
$1 \in A \mapsto id_{M}$.
\end{lemma} 

%
\begin{proof} We know that $e_M(D): \ _0\! \Ext_R^1(M,M) \to \ _0\!
  \Hom_R(I,A)$ is well-defined. To describe it, take any $\overline {\eta} \in
  \ _0\! \Ext_R^1(M,M)$ and let $\eta' \in \ _0\! \Hom(G^*,M)$ represent
  $\overline {\eta}$ and $\eta \in \ _0\! \Hom(G^*,F^*)$ map to $\eta'$, cf.
  \eqref{homext}. A maximal minor $f_1$ corresponds to choosing $t$ columns of
  $\cA$ which we for simplicity suppose is obtained by deleting the last $c-1$
  columns of $\cA$. With notations as in Sec.\! 2, $f_1$ is the determinant of
  $\varphi _1^*:G_1^*=\oplus _{j=0}^{t-1}R(-a_j) \to F^*=\oplus
  _{i=1}^tR(-b_i)$. By definition $e_M(D)(\overline {\eta})$ is given by the
  ideal generated by maximal minors of $\varphi^*+\epsilon \eta$. Let $I_1$ be
  the principal ideal generated by $f_1$. Since it is easy to compute
  $\det(\varphi_1^*+\epsilon \eta_1)$ where $\eta_1$ is the composition,
  $G_1^* \hookrightarrow G^* \stackrel {\eta}{ \longrightarrow} F^*$, of the
  natural inclusion $G_1^* \hookrightarrow G^*$ by $\eta$, we get that the
  image of $e_M(D)(\overline {\eta})$ in $_0\! \Hom_R(I_1,A) \simeq A_{(\deg
    f_1)}$ via the natural map $I_1 \hookrightarrow I$ is $tr(\varphi_1^a
  \cdot \eta_1) \otimes _R 1$ where $1 \in A$ is the unity, $\varphi_1^a$ is
  the adjoint of $\varphi_1^*$ and $tr$ the trace map.

  Let $M_1 = \coker \varphi_1^*$ and let $\pi:M_1 \to M$ be the canonical
  surjection. To see that $\delta_0(M)=i_{A,M} \cdot e_M(D)$ it suffices to
  check that the image of $\delta_0(M)(\overline {\eta})$ via the
  composition $$ _0\! \Hom_R(I,\Hom(M,M)) \twoheadrightarrow \ _0\!
  \Hom_R(I_1,\Hom(M,M)) \hookrightarrow \ _0\! \Hom_R(I_1,\Hom(M_1,M))$$ is
  $tr(\varphi_1^a \cdot \eta_1) \otimes _R \pi$ because the maximal minors
  provide a surjection $ \oplus R(-n_i) \twoheadrightarrow I$ and hence an
  injection $ _0\! \Hom_R(I,\Hom(M,M)) \to \ _0\! \Hom_R( \oplus
  R(-n_i),\Hom(M,M))$ (and think of $f_1$ as an arbitrary minor). Then we can
  conclude the proof by the functoriality of connecting homomorphisms and
  Ile's result for $\delta_0(M_1)$ in \cite{Icomp}. Indeed we have a
  commutative diagram
\[
\begin{array}{cccccc}
  _0\!\Ext_R^1(M,M) & \stackrel {\delta_0(M)}{
    \longrightarrow} &     _0\! \Hom_R(I,\Hom_R(M,M))
  \\ \downarrow  & &
  \downarrow  &     \\  _0\!\Ext_R^1(M_1,M)) &  \longrightarrow &
  _0\! \Hom_R(I_1,\Hom_R(M_1,M)) 
  \\ \uparrow  & & \uparrow   \\
  _0\!\Ext_R^1(M_1,M_1) & \stackrel {\delta_0(M_1)}{
    \longrightarrow} &     _0\! \Hom_R(I_1,\Hom_R(M_1,M_1)) \,  
\end{array}
\]
of three connecting homomorphisms. If $A_1=R/I_1$ then
$\delta_0(M_1)(\overline {\eta_1})= tr(\varphi_1^a \cdot \eta_1) \otimes _R
id_{M_1}$ by \cite[Prop.\! 2]{Icomp}, whence maps to $tr(\varphi_1^a \cdot
\eta_1) \otimes _R \pi$ via the right up-arrow. Here $\overline {\eta_1} \in \
_0\! \Ext_R^1(M_1,M_1)$ is represented by the composition 
$G_1^* \stackrel {\eta_1}{ \longrightarrow} F^* \to M_1$. Since the elements
$\overline {\eta} \in \ _0\! \Ext_R^1(M,M)$ and $\overline {\eta_1} \in \
_0\! \Ext_R^1(M_1,M_1)$ map to the same element in $ _0\! \Ext_R^1(M_1,M)$
(the one represented by the composition 
$G_1^* \stackrel {\eta_1}{ \longrightarrow} F^* \to M$)
the proof is complete.
\end{proof}

First we will find the dimension of the pro-representing object of
${\rm Def}_{M/R}$, i.e. we need to generalize \cite[Thm.\! 3.2]{K10} by
weakening its conditions so that it applies to an artinian $A$. Indeed we have
(and see  Proposition \ref{autdim} and Remark \ref{dimWcA} for computing the dimension).

\begin{theorem} \label{modulethm} Let $A=R/I_t(\cA)$ be standard determinantal
  and let $M = \coker \varphi^*$. Then $M$ is unobstructed, i.e. ${\rm
    Def}_{M/R}$ is formally smooth. Moreover if \ $_0\!\Hom_A(M,M) \simeq k$,
  then ${\rm Def}_{M/R}$ is pro-representable, and the
  pro-representing object $H(M/R)$ of ${\rm Def}_{M/R}$
  satisfies 
   $$\dim H(M/R)=  \dim \, _0\! \Ext_R^1(M,M) = \lambda_c + K_3+K_4+...+K_c\ .$$
   Moreover we have \
   $\dim\, _0\! \Ext_R^1(M,M) = \sum_{j=0}^{t+c-2} H_M(a_j)- \sum_{i=1}^{t}
   H_M(b_i) +1\, .$
\end{theorem}

\begin{proof} For the unobstructedness of $M$, see \cite[Thm.\! 3.1]{K10} or
  Ile's PhD thesis \cite[ch.\! 6]{I01}.

 To see the dimension formula we {\it claim} that there is an exact sequence
\begin{equation} \label{homext}
  0 \to  \ _0\! \Hom_R(M,M) \to\ _0\! \Hom_R(F^*,M) \to \ _0\!
  \Hom_R(G^*,M) \to \ _0\! \Ext_R^1(M,M) \to 0.
\end{equation}
Indeed the map $d_1: \wedge^{t+1}G^* \otimes S_{0}(F)\otimes \wedge^tF \to
G^*$ appearing in the Buchsbaum-Rim complex \eqref{BR}, takes an element of
$\wedge^{t+1}G^* \otimes S_{0}(F)\otimes \wedge^tF $ to a linear combination
of maximal minors with coefficients in $G^*$ because
$I_t(\cA) 
=\im(\wedge^{t}G^* \otimes S_{0}(F)\otimes \wedge^tF \to R$) is generated by
maximal minors, whence $\ _0\! \Hom_R(d_1,M)=0$ because $I={\rm ann}(M)$. So
if we apply $\ _0\! \Hom_R(-,M)$ to \eqref{BR}, we get \eqref{homext} by the
definition of $\ _0\! \Ext_R^i(M,M)$.

Taking dimensions of the groups in \eqref{homext} and
using 
\eqref{gradedmorfismo} and $\ _0\! \hom(M,M)=1$ we get the latter dimension
formula for \ $\dim \ _0\! \Ext_R^1(M,M)$. Moreover if we apply the exact
functors $ _0\! \Hom_R(F^*,-)$ and $ _0\! \Hom_R(G^*,-)$ onto \eqref{defMi}
with $i=c$ we get
\[ \ _0\! \hom_R(G^*,M) - \ _0\! \hom_R(F^*,M) \ = \lambda_c - 1 + \ _0\!
\hom_R(G^*,B^*) - \ _0\! \hom_R(F^*,B^*) \]
by using the definition \eqref{lamda} of $ \lambda_c$. Hence 
we get the dimension formula (i.e. the rightmost
equality) of Theorem~\ref{modulethm} provided we can prove $$ \ _0\!
\hom_R(B,G) - \ _0\! \hom_R(B,F) = K_3+...+K_c.$$ This follows from
Proposition~\ref{autdim} and the second exact sequence of Lemma~\ref{key}.

Finally the condition $_0\!\Hom_A(M,M) \simeq k$ allows us to lift
automorphisms, i.e. $ {\rm Def}_{M/R}$ is pro-representable, cf. \cite[Thm.\!
19.2]{Hart1}, whence $\dim H(M/R )= \dim \ _0\! \Ext_R^1(M,M)$ by the
smoothness of ${\rm Def}_{M/R}$.
\end{proof}

Let $D:=k[\epsilon]/(\epsilon^2)$ be the dual numbers and denote the dimension
in Theorem~\ref{modulethm} by 
\begin{equation} \label{lamext}
\lambda := \dim \ _0\! \Ext_R^1(M,M) =
\lambda_c + K_3+K_4+...+K_c \, .
\end{equation}
 Recalling that $
W_s(\underline{b};\underline{a})$ is a certain quotient of an open irreducible
set in the affine scheme $\VV =\Hom_{R}(G^*,F^*)$ parameterizing determinantal
$k$-algebras (Proposition \ref{WWs}),
  we get

  \begin{theorem} \label{compthm} Let $A=R/I$, $I=I_{t}(\cA)$, be a standard
    determinantal $k$-algebra and let $M = \coker \varphi^*$. \\[-2mm]

    {\rm (i)} If \ $_0\!\Hom_A(M,M) \simeq k$ and $ _0\! \Ext_A^1(M,M)= 0$
    then 
$${\rm Def}_{A \in
  W_s(\underline{b};\underline{a})} \simeq {\rm Def}_{M/R} \ . $$
Hence ${\rm Def}_{A \in W_s(\underline{b};\underline{a})}$ is a formally
smooth pro-representable functor and the pro-representing object has dimension
$$\dim   W_s(\underline{b};\underline{a}) = \lambda_c + K_3+K_4+...+K_c\ ,$$
cf.\! Remark \ref{dimWcA}. Moreover the tangent space of \
${\rm Def}_{A \in W_s(\underline{b};\underline{a})}$ is the subvector space of
$ _0\! \Hom_R(I,A)$ that corresponds to graded deformation $R_D \to A_D$ of
$R \to A$ to $D$ of the form $A_D=R_D/I_t(\cA_D)$ for some matrix $\cA_D$
which lifts $\cA$ to $D$. \vspace{0.15cm}

{\rm (ii)} If in addition \ $ _0\! \Ext_A^2(M,M)= 0$, then $ {\rm Def}_{M/R}
\simeq {\rm Def}_{A \in W_s(\underline{b};\underline{a})} \simeq {\rm
  Def}_{A/R}$ and \ ${\rm Def}_{A/R}$ is formally smooth. Moreover every
deformation of $A$ 
comes from deforming $\cA$ (cf.\! Definition~\ref{everydef}).
  \end{theorem}

  \begin{proof} We already know that \
    ${\rm Def}_{M/R}(T) \to {\rm Def}_{A \in
      W_s(\underline{b};\underline{a})}(T)$
    is well-defined (Fitting's lemma) and obviously surjective for any
    $(T,\mathfrak{m}_T)$ in $\underline{\ell}$. To see the injectivity, let
    $M_i$ for $i=1,2$ represent two elements of ${\rm Def}_{M/R}(T)$ mapping
    to the same element
    $A_T \in {\rm Def}_{A \in W_s(\underline{b};\underline{a})}(T)$. Since
    $M_i$, $i=1,2$ are cokernels of two morphisms given by matrices
    $ (\cA_T)_1$ and $ (\cA_T)_2$ lifting $\cA$ to $T$, we have
    $A_T=R_T/I_t((\cA_T)_1) = R_T/I_t((\cA_T)_2)$, cf.\!
    Definition~\ref{subf}. To show the injectivity, we may by induction on
    $r \ge 1$ suppose $\mathfrak{m}_T^{r+1}=0$ and
    ${\rm Def}_{M/R}(T/\mathfrak{m}_T^{r}) \rightarrow {\rm Def}_{A \in
      W(\underline{b};\underline{a})}(T/\mathfrak{m}_T^{r})$
    injective. Noticing that the condition $_0\!\Hom_A(M,M) \simeq k$ allows
    us to lift automorphisms of $M_i \otimes_T T/\mathfrak{m}_T^{r}$ to
    automorphisms of $M_i$, it follows from $ _0\! \Ext_A^1(M,M)= 0$ that the
    graded deformations $M_1$ and $M_2$ represent the {\em same} element of
    ${\rm Def}_{M/R}(T)$, and the injectivity is proved.

    Moreover using
    ${\rm Def}_{A \in W_s(\underline{b};\underline{a})} \simeq {\rm
      Def}_{M/R}$
    and Theorem~\ref{modulethm} we get that
    ${\rm Def}_{A \in W_s(\underline{b};\underline{a})} $ is pro-representable
    and formally smooth, whence its pro-representing object $H$ satisfies
    $H \simeq H(M/R)$ and since $H$ and its ``universal family'' $A_H$ are
    algebraizable by the proof of \cite[Thm.\! 5.2]{K10} (or use the $2^{nd}$
    paragraph of the proof of Theorem~\ref{compthmvar} of this paper to see
    that we get a morphism $\cO_{\GradAlg(H),(A)} \twoheadrightarrow S$,
    $H \simeq \hat S $ with ``universal family'' $A_S:=R_S/I_t(\cA_S)$, which
    extends to open subsets $V \subset U$, of
    $ W_s(\underline{b};\underline{a})$ and $ \GradAlg(H)$ respectively with
    $(A) \in V$ and $\cO_{V,(A)} =S$), we get
    $\dim H = \dim W_s(\underline{b};\underline{a})$ by
    Theorem~\ref{modulethm}. The description of its tangent space follows from
    \eqref{specseq} and Definition~\ref{subf} since $ _0\! \Hom_R(I,A)$ is the
    tangent space of $ {\rm Def}_{A/R}$.

If $ _0\! \Ext_A^2(M,M)= 0$ then
    $_0\!\Ext_R^1(M,M) \simeq \ (E_2^{0,1})_0\simeq \ _0\!
    \Hom_R(I,\Hom_A(M,M))$ by \eqref{specseq}. Hence the first part of the
    proof and $A \hookrightarrow \Hom_A(M,M)$ imply that all the injections in
    \begin{equation} \label{specseq2} {\rm
        Def}_{M/R}(D)   \hookrightarrow  \ {\rm
        Def}_{A/R}(D) \simeq \ _0\! \Hom_R(I,A)  \hookrightarrow \ _0\!
      \Hom_R(I,\Hom_A(M,M)) 
\end{equation}
are isomorphisms of finite dimensional vector spaces. It follows that $ {\rm
  Def}_{M/R} \to {\rm Def}_{A/R}$ is an isomorphism since it is bijective on
tangent spaces and $ {\rm Def}_{M/R}$ is formally smooth. Then we conclude the
proof by arguing as in the proof of \cite[Thm.\! 5.2]{K10}.
\end{proof}
\begin{remark} \label{Amodulerem3} The theorems of this section admit {\it
    substantial generalizations} with respect to $R$ being a polynomial ring.
  Indeed we may let $R$ be any graded quotient of a polynomial ring. 
  The main reason for this 
  is that the spectral sequence of this section, cf.\;\eqref{specseq}, Fitting's
  lemma and the exactness of the Buchsbaum-Rim complex are valid with almost
  no assumption on $R$ (but we need to replace the binomials defining
  $\lambda_c$ and $K_i$ with their Hilbert functions; the final formula of
  Theorems~\ref{modulethm} is, however, valid as stated).
  \end{remark}

\section{the locus of determinantal k-algebras}

In this section we generalize \cite[Theorems 5.5 and 5.8]{K10} concerning
dimension and smoothness of $\GradAlg(H)$ along
$ W_s(\underline{b};\underline{a})$,
to cover the artinian case. Indeed using that
Theorem~\ref{compthm} extends \cite[Thm.\! 5.2]{K10} by only assuming
$_0\!\Hom_A(M,M) \simeq k$ and $A$ standard determinantal instead of good
determinantal, the generalizations in Theorems~\ref{Amodulethm3} and
\ref{Amodulethm5} from good to standard determinantal are rather immediate. In
this section we also prove a new result (Theorem~\ref{compthmvar}) to cover
cases where $\dim W_s(\underline{b};\underline{a}) \ne\lambda$.

In the first theorem we let, as in \cite{K10}, $$\ext^2(M,M):= \dim \ker (\
_0\! \Ext_A^2(M,M) \to \ _0\! \Ext_R^2(M,M)\,)\, , \ \ {\rm cf. \
  \eqref{specseq}}\, . $$ Clearly $\ext^2(M,M) \le \ _0\! \ext_A^2(M,M), $ and
{\em note} that we below may replace $\GradAlg(H)$ with $ \Hi ^p(\PP^{n})$ if
$n-c \ge 1$, cf. the text accompanying \eqref{Grad} for explanations and
notations.

\begin{theorem} \label{Amodulethm3} Let $A=R/I$, $I=I_{t}(\cA)$ be a standard
  determinantal $k$-algebra, i.e. $(A) \in W_s(\underline{b};\underline{a})$,
  let $M = \coker \varphi^*$ and suppose \ $_0\!\Hom_A(M,M) \simeq k$ and
  $ _0\! \Ext_A^1(M,M) = 0$. Then 
 $$\dim   W_s(\underline{b};\underline{a}) =  \lambda:=\lambda_c +
 K_3+K_4+...+K_c \ , 
 $$cf.\! Remark \ref{dimWcA}. Moreover, for the codimension of $
 W_s(\underline{b};\underline{a})$
 in $\GradAlg(H)$ in a neighborhood of $(A)$ we have
 \begin{equation*} \label{codims} \dim_{(A)} \GradAlg(H) - \dim
   W_s(\underline{b};\underline{a}) \le \ \dim \ _0\!\Hom(I,A)   - \lambda  \le
   \ext^2(M,M)\, ,
 \end{equation*}
 where the first inequality is an equality if and only if $\GradAlg(H)$ is
 smooth at $(A)$. In particular these conclusions hold if \ $\dim A \ge 3
 +\dim R/I_{t-1}(\cA)$, {\rm \underline{or}} if \ $\dim A \ge 2$ and $\dim
 R/I_{t-1}(\cB) =0$ where ${\cB}$ is obtained from $\cA$ by deleting some
 column of ${\cA}$ (e.g. if \ $\dim A \ge 2$, $a_{i-2} \ge b_i$ for $2\le i\le
 t$ and $A$ is general).
\end{theorem}
\begin{proof} The $1^{\rm st}$ conclusion follows from
  Theorem~\ref{compthm}(i). For the statements on the codimension one knows
  that $ \dim_{(A)} \GradAlg(H) \le \dim\, _0\! \Hom_R(I,A)$ with equality if
  and only if $A$ is unobstructed (\cite[Thm.\! 1.5]{K79}). Using
  \eqref{specseq}, \eqref{lamext} and the injection
  $A \hookrightarrow\ _0\! \Hom_A(M,M)$, it follows that
  $ _0\! \Hom_R(I,A)-\lambda \le \dim\ (E_2^{0,1})_0 - \lambda = \ext^2(M,M)$,
  and we get the next conclusion of Theorem~\ref{Amodulethm3}. Now since
  $\depth_{I_{t-1}(\cA)A}A = \dim A -\dim R/I_{t-1}(\cA)$ and $A$ good
  determinantal is equivalent to $\depth_{I_{t-1}(\cA)A}A \ge 1$, the final
  sentence follows directly from \cite[Theorem 5.5]{K10}, and the parenthesis
  from Remark~\ref{dep}.
\end{proof}

There are many artinian determinantal rings satisfying the conditions
$_0\!\Hom_A(M,M) \simeq k$ and $ _0\! \Ext_A^1(M,M)= 0$ of
Theorem~\ref{Amodulethm3}. After having computed many examples using Macaulay
2 the general picture when $a_0 > b_t$ seems to be that the {\it only}
artinian determinantal rings $A=R/I_{t}(\cA)$ that do not satisfy these
conditions are those with matrix $\cA$ that is linear except in one column
${v} \in R_m^{\oplus t}$ where the degree is $m \ge 2$, so $\cA$ is of the
form $[\cB,{v}]$ where $\cB$ is linear. The following example avoids this
case, but $\cB$ is otherwise quite close to being linear (the case where $\cB$
is linear will be considered in Example~\ref{Blin}). Since the codimension-2
case is straightforward in this context, the first non-trivial case is $c=3$.
We include a codimension-4 example to see that it can be treated similarly.

\begin{example} \label{bo} (determinantal artinian quotients of $R$, using
  Theorem~\ref{Amodulethm3})

  (i) Let $R=k[x_0,x_1,x_2]$ and let $\cA= [\cB,v]$ be a general $2 \times 4$
  matrix with linear (resp. quadratic) entries in the first and second (resp.
  third) column and where the entries of $v$ are polynomials of the same
  degree $m$, i.e. the degree matrix of $\cA$ is
  $\left(\begin{smallmatrix}1 & 1  & 2  & m\\
      1 & 1 & 2 & m\end{smallmatrix}\right)$.
  The vanishing of all $2 \times 2$ minors defines an artinian ring with
  $h$-vector $(1,3,5,5,..,5,5,3)$, where the number of $5s$ is $m-1$. For
  $m \ge 2$ one verifies that the first conditions of
  Theorem~\ref{Amodulethm3}, i.e. that all conditions of
  Theorem~\ref{compthm}(i) hold, and it follows that
  $\overline{ W_s(\underline{b};\underline{a})}$ is an irreducible subset of
  $\GradAlg(H)$ of dimension $\lambda_3 + K_3$ which is $14$ for $m \ge 3$ (or
  one may use \eqref{dimWH} to find the dimension). For $m \ge 5$ one shows
  that 
  $_0\! \hom_A(I/I^2,A)=16$, 
  and we get that the codimension of
  $\overline{ W_s(\underline{b};\underline{a})}$ in $\GradAlg(H)$ is at most 2
  by
  Theorem~\ref{Amodulethm3}. 
  The dimensions of the $_0\! \Ext^i$-groups above are computed by using
  Macaulay 2 (over the finite fields $\ZZ_{101}$, $\ZZ_{701}$ and
  $\ZZ_{3001}$), and strictly speaking only for $m \le 10$, from which we
  clearly see the general pattern also for $m > 10$. Note that the dimensions
  we found of the $_0\! \Ext^i$-groups are independent of the characteristic
  of the fields and since field extensions are flat, they remain unchanged
  over any field containing one of these three fields. So the conclusions
  above hold at least when $k$ is the algebraic closure of $\ZZ_{p}$,
  $p \in \{101,701,3001 \}$.

  (ii) Let $R=k[x_0,x_1,x_2, x_3]$ and let $\cA= [\cB,v]$ be a general $2
  \times 5$ matrix 
  with degree matrix
  $\left(\begin{smallmatrix}1 & 1 & 1  & 2  & m\\
      1 & 1 & 1 & 2 & m\end{smallmatrix}\right)$.
  The vanishing of all $2 \times 2$ minors defines an artinian ring with
  $h$-vector $(1,4,7,7,..,7,4)$, where the number of $7s$ is $m-1$. For
  $m \ge 2$ one verifies that all conditions of Theorem~\ref{compthm}(i) hold,
  and we get that $\overline{ W_s(\underline{b};\underline{a})}$ is an
  irreducible subset of $\GradAlg(H)$ of dimension $\lambda_4 + K_3 +K_4$,
  which is equal to $25$ if $m \ge 3$ (or one may use \eqref{dimWH} to find
  the dimension). For $m \ge 5$, 
  $_0\! \hom_A(I/I^2,A)=33$, whence the codimension of
  $\overline{ W_s(\underline{b};\underline{a})}$ in $\GradAlg(H)$ is at most 8
  by
  Theorem~\ref{Amodulethm3}. 
  The computations of the $_0\! \Ext^i$-groups are verified by Macaulay 2 for
  $m \le 10$ over the fields $\ZZ_{101}$, $\ZZ_{701}$ and $\ZZ_{3001}$, but
  their dimensions are the same for any field containing one of these fields
  and for $m > 10$.
\end{example}

\begin{remark} \label{mac22} In all examples of this section we have verified
  the dimensions of those $ _0\! \Ext_A^i(M,M)$ and $_0\! \Ext_A^i(I/I^2,A)$,
  $i \ge0$ we need to apply Theorems~\ref{Amodulethm3}, ~\ref{Amodulethm5} and
  ~\ref{compthmvar} by using Macaulay 2 over each of the ground fields
  $\ZZ_{p}$ corresponding to $p = 101,701$ and $p = 3001$. But we have also
  computed these dimensions for many other values of $p > 10$, and it seems
  that they are independent of the characteristic of the field.
\end{remark}

Sometimes Theorem \ref{Amodulethm3} allows us to find the exact codimension of
$\overline{ W_s(\underline{b};\underline{a})}$ in $\GradAlg(H)$.
\begin{example} \label{boij}
(determinantal artinian quotients of $R=k[x_0,x_1,x_2]$, using
 Theorem~\ref{Amodulethm3})

 (i) Let $\cA$ be a general $2 \times 4$ matrix with quadratic (resp. linear)
 entries in the first (resp. second) row, i.e. with degree matrix   $\left(\begin{smallmatrix}2 & 2  & 2  & 2\\
     1 & 1 & 1 & 1\end{smallmatrix}\right)$.
 The vanishing of all $2 \times 2$ minors defines an artinian ring with
 $h$-vector $(1,3,6,4,1)$. Macaulay 2 computations (cf. Remark \ref{mac22})
 show that all conditions of Theorem~\ref{compthm}(i) hold,
 and it follows that $\overline{ W_s(\underline{b};\underline{a})}$ is an
 irreducible subset of $\GradAlg(H)$ of dimension $\lambda_3 + K_3=16$.
 Moreover we get that $A$ is unobstructed by \eqref{GradGenCI}. Indeed both
 $_0\! \Ext_A^1(I/I^2,A)=0$ (verified by Macaulay 2) and
 $ _0\!\Hom_A({\rm H}_2(R,A,A),A)=0$; the latter because
 $ _v{\rm H}_2(R,A,A) \hookrightarrow S_2(I)_v =0$ for $v < 6$ and the socle
 degree of $A$ is 4. 
 Since $_0\! \hom_A(I/I^2,A)=20$ by Macaulay 2, we get by
 Theorem~\ref{Amodulethm3} that
 $\codim_{\GradAlg(H)}\overline{ W(\underline{b};\underline{a})}=4$.

 (ii) Again $R=k[x_0,x_1,x_2]$, but now $\cA$ is a general $3 \times 5$ matrix
 with quadratic (resp. linear) entries in the first (resp. second and third)
 row. The vanishing of all $3 \times 3$ minors defines an artinian ring with
 $h$-vector $(1,3,6,10,5,1)$. Using Macaulay 2 
one verifies that all conditions of
 Theorem~\ref{compthm}(i) hold, and we get that
 $\overline{ W_s(\underline{b};\underline{a})}$ is an irreducible subset of
 $\GradAlg(H)$ of dimension $\lambda_3 + K_3=25$. Since
 $_0\! \Ext_A^1(I/I^2,A)=0$ by Macaulay 2 and
 $_v{\rm H}_2(R,A,A) \hookrightarrow S_2(I)_v =0$ for $v < 8$, we get
 $ _0\!\Hom_A({\rm H}_2(R,A,A),A)=0$, whence 
 $A$ is unobstructed by \eqref{GradGenCI}. Then Theorem~\ref{Amodulethm3} and
 $_0\! \hom_A(I/I^2,A)=40$ (by Macaulay 2) imply that
  $\codim_{\GradAlg(H)}\overline{ W(\underline{b};\underline{a})}=15$.
\end{example}

\begin{remark} Another possible way of finding the dimension of
  $\overline{ W_s(\underline{b};\underline{a})}$ and its codimension in
  $\GradAlg(H)$, as well as determining whether $\GradAlg(H_A)$ is smooth at
  $(A) \in \overline{ W_s(\underline{b};\underline{a})} \subset
  \GradAlg(H_A)$,
  is to consider some closely related and well-understood family of
  $\GradAlg(H_B)$ containing a point $(B)$ for which there is a surjection
  $B \to A$ of $k$-algebras whose Hilbert functions satisfy $H_B(v)=H_A(v)$
  for every $v \le s$ and $s$ sufficiently
  large. 
  Such a comparison is done in \cite{K09} when $\dim A \ge 1$, using
  $B=R/I_t(\cB)$, $A=R/I_t(\cA)$ and $\cA= [\cB,v]$, $v$ a column,
  cf.\;Example \ref{bo}. In \cite{K07} we explore this approach more
  generally, 
  e.g. see \cite[Thm.\;29]{K07} for a result where $A$ is artinian in
  $B \to A$.

  In the case $B=R/I_t(\cB)$ where $\cA= [\cB,v]$ and $v$ a column, we
  generalize in Sec.\;6 of this paper several results from \cite{K09}, see in
  particular Theorem \ref{Dim0Thm} where $\dim A = 1$. It is possible to
  extend this theorem to $A$ artinian by increasing the set of assumptions in
  Theorem \ref{Dim0Thm}, but since this theorem is already quite
  complicated as stated, we think Theorem \ref{Amodulethm3} is much simpler.
  It is also an option to replace $B=R/I_t(\cB)$ by an artinian Gorenstein
  quotient $A_g$ of $B$ such that $B \to A$ factors via $A_g \to A$ and, for
  $c=3$, use either \cite[Thm.\;2.3]{K98} or \cite[Sec. 5]{CV} to compute e.g.
  $\dim_{(A_g)} \GradAlg(H_{A_g})$. Applying, however, \cite[Thm.\;1]{K04} to
  $B \to A_g$ which states that $\GradAlg(H_{A_g})$ and $\GradAlg(H_{B})$ are
  much related, especially when $c=3$ by \cite[Example 28]{K04}, this latter
  approach is quite similar (concerning assumptions to be fulfilled) to the
  approach using $R/I_t(\cB) \to A$.
 \end{remark}

We have the following generalization of \cite[Theorem\! 5.8]{K10} in which the
good determinantal assumption ($\depth_{I_{t-1}(\cA)A}A \ge 1$) is weakened,
allowing $A$ to be artinian. {\em Note} that we below may replace
$\GradAlg(H)$ with $ \Hi ^p(\PP^{n})$ if $\dim A \ge 2$, cf. \eqref{Grad}.

\begin{theorem} \label{Amodulethm5} Let $A=R/I$, $I=I_{t}(\cA)$ be a standard
  determinantal $k$-algebra, i.e. $(A) \in W_s(\underline{b};\underline{a})$,
  let $M = \coker \varphi^*$ and suppose \ $_0\!\Hom_A(M,M) \simeq k$ and \
  $ _0\! \Ext_A^i(M,M) = 0$ for $i=1$ and $2$. Then $\GradAlg(H)$ is smooth at
  $(A)$ and 
 $$\dim_{(A)} \GradAlg(H) =  \lambda_c + K_3+K_4+...+K_c \ , $$
 (cf.\! Remark \ref{dimWcA}). Moreover
 $\overline {W_s(\underline{b};\underline{a})} \subset \GradAlg(H)$ is an
 irreducible component, and every deformation of $A$ comes from deforming
 $\cA$. In particular these conclusions hold
 if $\dim A \ge 4 +\dim R/I_{t-1}(\cA)$, {\rm \underline{or}} if \ $\dim A \ge
 3$ and $\dim R/I_{t-1}(\cB) =0$ where ${\cB}$ is obtained from $\cA$ by
 deleting some column of ${\cA}$ (e.g. if \ $\dim A \ge 3$, $a_{i-\min
   (3,t)}\ge b_{i}$ for $\min (3,t)\le i \le t$ and $A$ is general).
\end{theorem}

\begin{proof} The first two sentences of conclusions follow from
  Theorem~\ref{compthm} and Lemma~\ref{unobst}. Moreover since
  $\depth_{I_{t-1}(\cA)A}A = \dim A -\dim R/I_{t-1}(\cA)$, the
  final sentence follows directly from \cite[Theorem 5.8]{K10}, and the
  parenthesis from Remark~\ref{dep}.
\end{proof}
The condition $ _0\! \Ext_A^2(M,M)= 0$ of Theorem~\ref{Amodulethm5} seems harder
to satisfy for artinian rings, but it holds quite often if $\cA$ is not too
close to the linear case. Here is an example:

\begin{example} \label{bocomp} (determinantal artinian quotients of
  $R=k[x_0,x_1,x_2]$, using Theorem~\ref{Amodulethm5})

  (i) Let $\cA= [\cB,v]$ be a general $2 \times 4$ matrix with linear (resp.
  cubic) entries in the first and second (resp. third) column and where the
  entries of $v$ are polynomials of the same degree $m$, $m \ge 3$. The degree
  matrix of $\cA$ is
  $\left(\begin{smallmatrix}1 & 1  & 3  & m\\
      1 & 1 & 3 & m\end{smallmatrix}\right)$.
  The vanishing of all $2 \times 2$ minors defines an artinian ring with
  $h$-vector $(1,3,5,7,7,..,7,5,3)$, where the number of $7s$ is $m-2$. For
  $m=3$ and $m \ge 5$ one verifies that all conditions of
  Theorem~\ref{compthm}(ii), i.e. the first conditions of
  Theorem~\ref{Amodulethm5} hold, and it follows that
  $\overline{ W_s(\underline{b};\underline{a})}$ is a generically smooth
  irreducible component of $\GradAlg(H)$ of dimension $\lambda_3 + K_3=18$
  (resp. $17$) for $m \ge 5$ (resp. $m=3$). The computations for the
  $_0\! \Ext^i$-groups above are verified by Macaulay 2 for $m \le 10$ (cf.
  Remark \ref{mac22}), but their dimensions hold also for $m > 10$. One also
  verifies that $\ _0\! \ext_A^1(I/I^2,A) = 2$ for $6 \le m \le 10$, so both
  the generic smoothness (cf. \eqref{GradGenCI}) along
  $\overline{ W_s(\underline{b};\underline{a})}$ and the conclusion that
  $\overline{ W_s(\underline{b};\underline{a})}$ is an irreducible component
  of $\GradAlg(H)$ would be hard to see without using
  Theorem~\ref{Amodulethm5}.

  (ii) Let $\cA$ be a general $2 \times 4$ matrix with degree matrix
  $\left(\begin{smallmatrix}3 & 3  & 3  & 3\\ 
      1 & 1 & 1 & 1\end{smallmatrix}\right)$ \ .
The vanishing of all $2 \times 2$ minors
  defines an artinian ring with $h$-vector $(1,3,6,10,9,7,3,1)$. Using
  Macaulay 2 one verifies that all conditions of Theorem~\ref{compthm}(ii)
  hold, (i.e. also $ _0\! \Ext_A^2(M,M)=0$), and we get that $\overline{
    W_s(\underline{b};\underline{a})}$ is a generically smooth irreducible
  component of $\GradAlg(H)$ of dimension $\lambda_3 +
  K_3=29$. 
  In this example $_0\! \Ext_A^1(I/I^2,A) \ne 0$ and $_0\!
  \hom_A(I/I^2,A)=29$, so it is not straightforward to see that $A$ is
  unobstructed, but it is, due to Theorem~\ref{Amodulethm5} which
  also contains additional information.
\end{example}

As indicated it seems that the {\it only} artinian determinantal rings
$A=R/I_{t}(\cA)$ in the case $a_0>b_t$ that do not satisfy the conditions of
Theorem~\ref{compthm}(i) are those with a matrix $\cA$ that is linear except
in one column ${v} \in R_m^{\oplus t}$ where the degree is $m \ge 2$. Moreover
if $m \ge 3$ we may even have
$$ _0\!\Hom_A(M,M) \simeq k \ , \ \  _0\! \Ext_A^1(M,M) \ne 0 \ {\rm and} \ \
_0\!\Ext_A^2(M,M) = 0\, ,$$
and the deformation functors of these rings are fully determined by our next
Theorem~\ref{compthmvar}. For {\em artinian linear} determinantal rings (i.e.
$m= 1$ above) all conditions of Theorem~\ref{compthm}(ii) seem to hold and
even more, they are rigid in several senses: $ _0\! \Hom_R(I,A)=0$, as well as
$ _0\! \Ext_R^1(M,M)= 0$, whence the linear case may be not so interesting.
Recalling \eqref{lamext}, i.e.
$\lambda := \dim \ _0\! \Ext_R^1(M,M) = \lambda_c + K_3+K_4+...+K_c$ and
Remark \ref{dimWcA}, we get

\begin{theorem} \label{compthmvar} Let $A=R/I$, $I=I_{t}(\cA)$ be a standard
  determinantal $k$-algebra, let $M = \coker \varphi^*$ and suppose \
  $_0\!\Hom_A(M,M) \simeq k$. If the map
  $ _0\! \Ext_A^2(M,M) \to \ _0\! \Ext_R^2(M,M)$ of \eqref{specseq} is
  injective, or weaker, if the tangent map
  $\ _0\! \Ext_R^1(M,M)\rightarrow \ _0\! \Hom_R(I,A)$ of the natural morphism
  of functors $e_M:{\rm Def}_{M/R} \rightarrow {\rm Def}_{A/R}$
 is surjective, then $e_M$ is smooth. Moreover ${\rm Def}_{M/A}$
  is smooth and its pro-representing object $H(M/A)$ satisfies
$$ \dim H(M/A)= \dim \, _0\! \Ext_A^1(M,M)= \lambda - \dim
W_s(\underline{b};\underline{a})\, .$$ Furthermore $ {\rm Def}_{A \in
  W_s(\underline{b};\underline{a})} = {\rm Def}_{A/R}$, $ {\rm Def}_{A/R}$ is
smooth, $\overline{ W_s(\underline{b};\underline{a})} \subset \GradAlg(H)$ is
an irreducible component of dimension $\lambda - \dim \, _0\! \Ext_A^1(M,M)$
and every deformation of $A$ comes from deforming $\cA$. In particular these
conclusions hold if $\dim A \ge 3 + \min\{\dim R/I_{t-1}(\cA)+1,\dim
R/I_{t-1}(\cB)\} $ where ${\cB}$ is obtained from $\cA$ by deleting some
column of ${\cA}$ (e.g. if \ $\dim A \ge 3$, $a_{i-\min (3,t)}\ge b_{i}$ for
$\min (3,t)\le i \le t$ and $A$ is general).
  \end{theorem}
  \begin{remark} We may consider ${\rm Def}_{M/A}$ as a fiber functor of
    $e_M:{\rm Def}_{M/R} \rightarrow {\rm Def}_{A/R}$, using trivial
    deformations of $A$. For a description of the obstruction maps of
    ${\rm Def}_{M/A}$ 
    for arbitrary $M$, see \cite{L86, IleTr}. Note that the assumption
    $_0\!\Hom_A(M,M) \simeq k$ implies the pro-representability of both
    ${\rm Def}_{M/A}$ and ${\rm Def}_{M/R}$, cf. end of proof of
    Theorem~\ref{modulethm}.
\end{remark}
\begin{proof}
  Firstly  we observe that the composition of the natural maps:
\begin{equation*}  _0\! \Ext_R^1(M,M)   \stackrel {e_M(D)}{
      \longrightarrow}  \ \! _0\! \Hom_R(I,A)
  \hookrightarrow \ _0\! \Hom_R(I,\Hom_A(M,M)) \simeq (E_2^{0,1})_0
\end{equation*}
described in Lemma~\ref{specFitt} is just the degree zero part of the map
$ \Ext_R^1(M,M) \to E_2^{0,1}$ appearing in \eqref{specseq}. This composition is
surjective 
if and only if the $_0\! \Ext_R^2$-map of Theorem
\ref{compthmvar} is injective by \eqref{specseq}. Since
$ A \rightarrow \Hom_A(M,M)$ is injective, it follows that the injectivity of
this $_0\! \Ext_R^2$-map is equivalent to the surjectivity of
$e_M(D) :\  _0\! \Ext_R^1(M,M) \to \! _0\! \Hom_R(I,A)$ and
$ _0\! \Hom_R(I,A) \simeq \ _0\! \Hom_R(I,\Hom_A(M,M))$ being bijective.

 Let $T \rightarrow S$ be a small artinian surjection with kernel
${\mathfrak a}$. 
  To prove the (formal) smoothness of $e_M:{\rm Def}_{M/R} \rightarrow {\rm
    Def}_{A/R}$, we must by definition show that the induced map
$${\rm Def}_{M/R}(T) \rightarrow {\rm Def}_{M/R}(S) \times_{{\rm
    Def}_{A/R}(S)} {\rm Def}_{A/R}(T)$$
is surjective. Let $M_S$ be an arbitrary fixed deformation of $M$ to $S$
inducing $R_S/I_S \in {{\rm Def}_{A/R}(S)}$, let $R_T/I_T$ be a deformation of
$R_S/I_S$ to $T$ and let $e_{M_S}$ be the functor that takes a deformation
$M_T= \coker(\varphi^*)$ of $M_S$ onto the deformation of $R_S/I_S$ defined by
the maximal minors of a matrix representing $\varphi$. Note that since $M$ is
unobstructed, there exists a deformation $M_T'$ of $M_S$ to $T$ inducing a
deformation $R_T/I_T'$ of $R_S/I_S$. The difference ``$R_T/I_T -
R_T/I_T'$''  
sits in $ _0\! \Hom_R(I,A) \otimes_k {\mathfrak a}$ and since $e_{M_S}(T)
\otimes_T id_{\mathfrak a}$ is up to isomorphism equal to 
the surjective map $e_M(D) \otimes_k {\mathfrak a}$ as one may see by
the arguments described in the first paragraph of the proof of
Lemma~\ref{specFitt}, $e_{M_S}(T) \otimes_T id_{\mathfrak a}$ is surjective.
Hence there exists an element in $ _0\! \Ext_R^1(M,M) \otimes_k {\mathfrak a}$
that we can ``add'' to $M_T'$ to get a deformation $M_T$ of $M_S$ that induces
$R_T/I_T$, i.e. $e_M$ 
is smooth. It follows that ${\rm Def}_{A/R}$ is smooth. We also get that
${\rm Def}_{A \in W_s(\underline{b};\underline{a})}(T) \hookrightarrow {\rm
  Def}_{A/R}(T)$
is surjective for any $T \in {\rm ob}(\underline {\ell})$, and injective by
Definition~\ref{subf}, whence an equality, and since ${\rm Def}_{M/A}$ is a
fiber functor of ${\rm Def}_{M/R} \rightarrow {\rm Def}_{A/R}$, it follows
that ${\rm Def}_{M/A}$ is smooth. Hence
$ \dim H(M/A)= \dim \ _0\! \Ext_A^1(M,M)$, and we get the displayed dimension
formulas of Theorem~\ref{compthmvar} from \eqref{specseq} and the surjectivity
of $ _0\! \Ext_R^1(M,M) \to \ \! _0\! \Hom_R(I,A)$ provided
$ \dim W_s(\underline{b};\underline{a}) = \ _0\! \hom_R(I,A)$.

If we compare $ {\rm Def}_{A \in W_s(\underline{b};\underline{a})} = {\rm
  Def}_{A/R}$ with Definition~\ref{subf} we get that every deformation $A_T$
of $A$ to an artinian $T$ comes from deforming $\cA$. To see that we may
replace ``an artinian $T$'' by ``$T$ a local ring'' in this statement
(cf.\;Definition~\ref{everydef}) we pick $d:= \ _0\! \hom_R(I,A)$
elements of $\ _0\! \Ext_R^1(M,M)$ that map to 
linearly independent elements of $\ _0\! \Hom_R(I,A)$ via the tangent map $\
_0\! \Ext_R^1(M,M)\twoheadrightarrow \ _0\! \Hom_R(I,A)$ of $e_M$, and we let
$\eta_1,...,\eta_d \in \ _0\! \Hom(G^*,F^*)$ with presentation matrices
$\cA_1,...,\cA_{d}$ correspond to the $d$ elements we picked in $\ _0\!
\Ext_R^1(M,M)$, cf. the beginning of the proof of Lemma~\ref{specFitt} for a
similar set-up. Let $\cA_S := \cA+ s_1
\cA_1+...+s_{d}\cA_{d}$ 
and $S:=k[s_1,...,s_{d}]_{(s_1,...,s_{d})}$ where $k[s_1,...,s_{d}]$ is a
polynomial ring, and let $O:=\cO_{\GradAlg(H),(A)}$ be the local ring of
$\GradAlg(H)$ at the $k$-point $(A) \in \GradAlg(H)$, $k=\overline k$.
Then the algebraic family $A_S:=R_S/I_t(\cA_S)$ is $S$-flat by
Lemma~\ref{defalpha} and e.g.\;using the explicit description of the
pro-representing object $H(A/R)$ of $ {\rm Def}_{A/R}$ appearing in the proof
of Thm.\;4.2.4 of \cite{L} 
to see that {\it any} deformation of $A_S \otimes_S S/(s_1,...,s_{d})^2$ to
$H(A/R)$ may serve as a versal lifting, we get that $H(A/R)$ equals $S$
with``universal object'' $A_S$ (up to completion, but notice that the versal
lifting is defined by polynomials, not power series, in $s_1,...,s_d$). Hence,
by the universal property of the representable functor that corresponds to
$\GradAlg(H)$, there is a morphism $O \to S$ whose completions are isomorphic
because both pro-represents ${\rm Def}_{A/R}$ on $\ell$ (\cite[(2,1) and
(2.8)]{S}). Thus $O \to S$ is injective, and since $O/{\mathfrak m}_O^n \simeq
S/{\mathfrak m}_S^n$ for every $n$, it follows that $O \to S$ is an
isomorphism, having $A_O:=A_S \otimes_S O$ as the pullback (considered as
scheme) of the universal object of $ \GradAlg(H)$ to $\Spec(O)$. This shows
that ``every deformation of $A$ comes from deforming $\cA$`` because the
universal property implies that all deformations are given by pullback to
$\Spec(T)$ (i.e. by taking tensor product of $O\to A_O$ via $O\to T$). Then
Lemma~\ref{unobst} implies that $\overline{ W_s(\underline{b};\underline{a})}
\subset \GradAlg(H)$ is a generically smooth irreducible component of
dimension $d$, and in particular we have $ \dim
W_s(\underline{b};\underline{a}) = \ _0\! \hom_R(I,A)$.

The argument for the statement in the final sentence is as in the proof of
Theorem~\ref{Amodulethm5} provided $\dim A \ge 4 + \dim
R/I_{t-1}(\cA)$. 
If, however, $\dim A \ge 3 + \dim R/I_{t-1}(\cB)$ then
$\depth_{I_{t-1}(\cB)}R/I_{t}(\cB) \ge 4$ and we get $ _0\! \Ext_A^2(M,M)=0$
by \cite[Thm.\! 4.5]{K10} and we are done.
\end{proof}

As a corollary to the $3^{\rm rd}$ paragraph of the proof, we get that if
every graded deformation of $A$ to any local artinian $k$-algebra
comes from deforming $\cA$, then every deformation of $A$ comes from deforming
$\cA$ (see Definition~\ref{everydef}). But the proof implies more:

\begin{corollary} \label{everydefT} Let $A=R/I_{t}(\cA)$ be a standard
  determinantal $k$-algebra. If every graded deformation $A_D$ of $A$ to the
  dual numbers $D:=k[\epsilon]/(\epsilon^2)$ comes from deforming $\cA$, or
  equivalently, if the tangent map of the natural morphism
  $e_M:{\rm Def}_{M/R} \rightarrow {\rm Def}_{A/R}$,
  $$e_M(D): {\rm Def}_{M/R}(D) = \ _0\! \Ext_R^1(M,M) \to {\rm Def}_{A/R}(D)=
  \ _0\! \Hom_R(I,A) \ , $$
  is surjective, then every deformation of $A$ comes from deforming
  $\cA$. 
\end{corollary} 

\begin{proof} Since every graded deformation of $M$ comes from deforming $\cA$
  by Lemma \ref{defalpha}, it is clear that ``every graded deformation of $A$
  to $D$ comes from deforming $\cA$'' is equivalent to the surjectivity of the
  tangent map $e_M(D)$. Then we get what we want from the proof appearing in
  the $3^{\rm rd}$ paragraph of Theorem~\ref{compthmvar} since that part of
  the proof only requires the surjectivity of $e_M(D)$ and not the condition
  $_0\!\Hom_A(M,M) \simeq k$.
\end{proof} 

Let us consider two particular cases of artinian rings:

\begin{example} \label{Blin} (determinantal artinian quotients of $R$, using
  Theorem~\ref{compthmvar})

  {\rm (i)} Let $R=k[x_0,x_1,x_2]$ and let $\cA= [\cB,v]$ be a general
  $2 \times 4$ matrix where $\cB$ is linear and the entries of $v$ are
  polynomials of the same degree $m \ge 2$. The vanishing of all $2 \times 2$
  minors defines an artinian ring with $h$-vector $(1,3,3,..,3)$ where the
  number of $3s$ is $m$. Using Macaulay 2 for $3 \le m \le 10$ one verifies
  (cf.\;Remark \ref{mac22}) that all conditions of Theorem~\ref{compthmvar}
  hold and that $\ _0\! \ext_A^1(M,M)=2$, and it follows that
  $\overline{ W_s(\underline{b};\underline{a})}$ is a generically smooth
  irreducible component of $\GradAlg(H)$ of dimension
  $\lambda_3 + K_3-\ _0\! \ext_A^1(M,M)=6$ 
  (true also for $m>10$). If $m=2$, Theorem~\ref{compthmvar} does not apply
  because $\ _0\! \Ext_A^2(M,M) \ne 0$. In fact one may verify (Macaulay 2)
  that all inclusions of \eqref{specseq2} are strict, i.e. non-isomorphisms.

  {\rm (ii)} Let $R=k[x_0,x_1,x_2,x_3]$ and let $\cA= [\cB,v]$ be a general
  $2 \times 5$ matrix where $\cB$ is linear and the entries of $v$ are
  polynomials of the same degree $m \ge 2$. The vanishing of all $2 \times 2$
  minors defines an artinian ring with $h$-vector $(1,4,4..,,4)$, where the
  number of $4s$ is $m$. Using Macaulay 2 for $3 \le m \le 10$ one verifies
  that all conditions of Theorem~\ref{compthmvar} hold and that
  $\ _0\! \ext_A^1(M,M)=4$. By Proposition~\ref{autdim} and \eqref{lamda},
  $K_3=0, K_4=4 $ and $\lambda_4=12$, and it follows that
  $\overline{ W_s(\underline{b};\underline{a})}$ is a generically smooth
  irreducible component of $\GradAlg(H)$ of dimension
  $\lambda_4 + K_3+K_4-\ _0\! \ext_A^1(M,M)=12$ (true also for $m>10$).
\end{example}

\begin{remark} \label{ghext2} Comparing Theorem \ref{compthmvar} with Theorem
  \ref{Amodulethm5} we see that the main condition for
  $\overline{ W_s(\underline{b};\underline{a})}$ to be an irreducible
  component of $\GradAlg(H_A)$ is that $ _0\! \Ext_A^2(M,M)$ (or
  $\ext^2(M,M)$, cf. Theorem \ref{Amodulethm3}) vanishes; the condition
  $ _0\! \Ext_A^1(M,M)=0$ of Theorem \ref{Amodulethm5} is in this regard less
  important. An interesting observation to make to all examples of this
  section is that ghost terms in the minimal resolution of $A$ appear
  precisely when $_0\! \Ext_A^2(M,M) \ne 0$. Indeed for all $A$ in Examples
  \ref{bo} and \ref{boij} we have $_0\! \Ext_A^2(M,M) \ne 0$ (Macaulay 2) as
  well as ghost terms in the minimal resolution of $A$ while we in Examples
  \ref{bocomp} and \ref{Blin} where
  $\overline{ W_s(\underline{b};\underline{a})}$ is a component due to
  $ _0\! \Ext_A^2(M,M) = 0$, have no ghost terms in the minimal resolution
  of $A$. So one may wonder if there is a connection between
  $ _0\! \Ext_A^2(M,M) \ne 0$ and the existence of ghost terms for small $c$
  in the artinian case (and maybe also when $\dim A \le 1$, cf. Example
  \ref{Llin}, but not when $\dim A \ge 2$, cf. Example \ref{gh2}). Indeed we
  devote the next section to a study of ghost terms in the Eagon-Northcott
  resolution (\ref{ENres}).
 \end{remark}

 For determinantal zero-schemes we showed in \cite{K09} that the dimension
 formula \eqref{dimW} fails when (and only when?) $\cA$ is a linear matrix
 consisting of two rows, cf. Remark~\ref{dim0new}(ii). In our next
 example we use Theorem~\ref{compthmvar} to treat this case.

 \begin{example} \label{Llin} (determinantal quotients of
   $R=k[x_0,x_1,\cdots,x_c]$ of dimension one)

  Let $\cA$ be a general $2 \times (c+1)$ matrix of linear entries. The
  vanishing all $2 \times 2$ minors defines a reduced scheme $X=\Proj(A)$ of
  $c+1$ general points in $\PP^{c}$. Since we may suppose $A$ is good
  determinantal, $\Hom_A(M,M) \simeq A$ by \cite[Lem.\! 3.2]{KM}, and
  we have $ _0\! \ext_R^1(M,M)=c(c+1)+c-2$ by Theorem~\ref{modulethm}. In the
  range $3 \le c \le 10$ we use Macaulay 2 to see \ $ _0\! \ext_A^1(M,M)=c-2$
  and \ $ _0\! \Ext_A^2(M,M)=0$. For $c \le 10$ we get from
  Theorem~\ref{compthmvar} that $\overline {W({0,0};{1,1,...,1})} $ is a
  generically smooth irreducible component of $\GradAlg(H)$ of dimension
  $c(c+1)$ and that every deformation of $A$ comes from deforming $\cA$.

  It is possible to describe $\overline {W({0,0};{1,1,...,1})} $ for all
  $c \ge 2$ without using Macaulay 2. Indeed it is well known that there exists
  a generically smooth component of the usual Hilbert scheme
  $\Hi ^p(\PP^{c})$, of dimension $c(c+1)$, containing an open subset
  parametrizing $c+1$ reduced points in $\PP^{c}$. Since the Hilbert function
  of $X$ is $(\dim A_i)=(1,c+1,c+1,...)$, we get
  $H^1({\mathcal I}_X(v)) \simeq \ _vH^{1}_{\goth m}(A) = 0$ for $v\ge 1$ and
  hence $\ _0\!\Hom_R (I_X,H^{1}_{\goth m}(A)) = 0$. $\GradAlg(H)$ is
  therefore isomorphic to $\Hi ^p(\PP^{c})$ at $(X)$ by \eqref{Grad}, and it
  follows that $\overline {W({0,0};{1,1,...,1})} $ is a generically smooth
  irreducible component of $\GradAlg(H)$ of dimension $c(c+1)$.
\end{example}

Theorem~\ref{compthmvar} admits a nice consequence concerning glicciness.
Indeed since glicciness is not necessarily an open property, the following
result may be useful.

\begin{corollary} \label{glicci} Let $X=\Proj(A)$, $I_X=I_{t}(\cA)$, be a
  standard determinantal scheme and suppose that $\ _0\!\Hom_R
  (I_X,H^{1}_{\goth m}(A)) = 0$ and that the map $ _0\! \Ext_A^2(M,M) \to \
  _0\! \Ext_R^2(M,M)$ of \eqref{specseq} is injective. Then the Hilbert scheme
  \ $\Hi ^p(\PP^{n})$ is smooth at $(X)$ and $(X)$ belongs to a unique
  irreducible component of $\Hi ^p(\PP^{n})$ whose general element $\tilde X
  \subset \PP^{n}$ is glicci. In particular this conclusion holds if $\dim X
  \ge 2 + \min\{\dim R/I_{t-1}(\cA)+1,\dim R/I_{t-1}(\cB)\} $ where ${\cB}$ is
  obtained from $\cA$ by deleting some column of ${\cA}$.
  \end{corollary}
 
  \begin{proof} Indeed $\tilde X$ is standard determinantal by
    Theorem~\ref{compthmvar} and \eqref{Grad} and since standard determinantal
    schemes are glicci by \cite[Thm.\! 3.6]{KMMNP}, we are done.
\end{proof}

Finally we generalize \cite[Thm.\! 5.16]{K10} which is about codimension-2
determinantal quotients of an ACM scheme,  
to cover the artinian case. 
%
Below $ \overline R$ is a Cohen-Macaulay quotient of a polynomial ring
(i.e.\;$\Proj( \overline R) \subset \PP_k^{N}$ is ACM) and $A$ is a standard
determinantal quotient of $ \overline R$, but notice that it is not
necessarily a determinantal quotient of $ k[x_0, \cdots,x_N] $. Moreover
letting $\underline{b}$, $\underline{a}$ be as in \eqref{ba} and
$(A) \in W_s(\underline{b};\underline{a})$ observe that, for $\dim A \ge 2$,
we may replace $\GradAlg^H( \overline R)$ by $ \Hi ^p(\Proj(\overline R))$
since \eqref{Grad} extends to hold in this generality by \cite[Thm.\! 3.6,
Rem.\! 3.7]{K79}. For the irreducibility and dimension of
$\overline{ W_s(\underline{b};\underline{a})}$ in the case
$\overline R = k[x_0,x_1]$, see also \cite[Thms.\! 2.9, 2.12, 3.13]{I}.
  \begin{theorem} \label{elling} Let $P=\Proj( \overline R) \subset \PP_k^{N}$
    be an ACM scheme where $k$ is any field and let $X=\Proj(A) \subset P$,
    $A=R/I_{t}(\cA)$, be any standard determinantal scheme of codimension $2$
    in $P$. Then $\GradAlg^H( \overline R)$ is smooth at $(A)$ and \ $\dim_{(A)}
    \GradAlg^H( \overline R) = \lambda( \overline R)_2$ where
 $$ \lambda( \overline R)_2:= \sum_{i,j} \dim  \overline R_{(a_j-b_i)}
 + \sum_{i,j} \dim \overline R_{(b_i-a_j)} - \sum _{i,j} \dim \overline
 R_{(a_i-a_j)}- \sum _{i,j} \dim \overline R_{(b_i-b_j)} + 1 . $$ Moreover
 every deformation of $A$ comes from deforming $\cA$. In particular if
 $k=\overline{k}$, then $\GradAlg^H( \overline R)$ is smooth along
 $W_s(\underline{b};\underline{a})$ and the closure $\overline{
   W_s(\underline{b};\underline{a})}$ in $\GradAlg^H( \overline R)$ is an
 irreducible component of dimension $ \lambda( \overline R)_2$.
\end{theorem}

\begin{proof} Since $M$ is a twist of the canonical module of $A$ when $c=2$,
  it is well known that $ _0\! \Ext_A^i(M,M) = 0$ for $i>0$ and
  $\Hom_A(M,M) \simeq A$. We get
  $ {\rm Def}_{M/ \overline R} \simeq {\rm Def}_{A/ \overline R}$ and every
  statement of the theorem from Theorem~\ref{compthm}, Lemma~\ref{unobst} and
  Remark~\ref{Amodulerem3}, except the displayed formula. Using, however, the
  last conclusion of Theorem~\ref{modulethm}, we easily prove the displayed
  formula, cf. \eqref{BR}. Finally note that we have included the assumption
  $k=\overline{k}$ in the statements of ${ W_s(\underline{b};\underline{a})}$
  for the only reason that ${ W_s(\underline{b};\underline{a})}$ is by
  definition a certain locus (of $k$-points) of $\GradAlg^H( \overline R)$.
\end{proof}
So there are no singular points of $\GradAlg^H( \overline R)$ at $(A) \in
W_s(\underline{b};\underline{a})$ when $c=2$ while singular points of
$\GradAlg^H( \overline R)$ for $c>2$ are quite common (see \cite[Rem.\!
3.6]{KM09} and \cite{siq}).

\section{ghost terms}

Let $N$ be a graded $R$-module with a {\it minimal} $R$-free resolution $0
\to P_{\bullet} \twoheadrightarrow N$. By a ghost term of $P_{\bullet}$ we
mean a direct free summand that appears in consecutive terms of $P_{\bullet}$.
As in the minimal resolution conjecture (\cite{lor, M}, and see \cite{MMN,
  MP,ML} and its references for other contributions), one expects that a
general element $R/I$ of $\GradAlg(H)$ contains ``as few ghost terms as
possible'', while ghost terms for special elements of $\GradAlg(H)$ are more
common. Indeed the semicontinuity of graded Betti numbers imply that these
numbers decrease under generization (deformation to a more general
element), but there may still be ghost terms left in the minimal free
resolution of the general element of $\GradAlg(H)$. In this section we shall
see that some ghost terms of a determinantal ring $R/I$ easily disappear under
generizations while other ghost terms do not.
Letting 
$\underline{a}_{\hat i} = a_0,...a_{i-1},a_{i+1}..., a_{t+c-2}$ and
$\underline{b}_{\hat i} = b_1,...b_{i-1},b_{i+1}..., b_{t}$ we have the
following result (allowing $t=2$, in which case $W_s(\underline{b}_{\hat
  i};\underline{a}_{\hat j})$ consists of complete intersections of $R$).

\begin{proposition}\label{ghost} \ Suppose $a_j = b_i$ for some $j
  \in \{0,t+c-2\}$
  and $i \in \{1,t\}$. Then we have an inclusion
  $ W_s(\underline{b}_{\hat i};\underline{a}_{\hat j}) \subset
  W_s(\underline{b};\underline{a}) $
  of open irreducible subsets of $\GradAlg(H)$. Moreover if
  $W_s(\underline{b};\underline{a}) \setminus W_s(\underline{b}_{\hat
    i};\underline{a}_{\hat j}) \ne \emptyset$,
  then every $R/I$ of
  $W_s(\underline{b};\underline{a}) \setminus W_s(\underline{b}_{\hat
    i};\underline{a}_{\hat j})$
  admits a generization
  $(R/I_g) \in W_s(\underline{b}_{\hat i};\underline{a}_{\hat j})$ removing
  exactly all ghost terms in the Eagon-Northcott resolution \eqref{ENres}
  coming from $a_j = b_i$; in particular a resolution of $R/I_g$ of
  $W_s(\underline{b}_{\hat i};\underline{a}_{\hat j})$ is given by its
  Eagon-Northcott
  resolution. 
 \end{proposition}

 \begin{proof} As observed in the proof of Lemma~\ref{WWs}, the
   vanishing of $\Ext_R^i(R/I_{t}(\cA),R)$ is an open property. This holds
   also for the elements $R/I_{t}(\cA)$ of $\GradAlg(H)$, i.e. the subset of
   $\GradAlg(H)$ such that $I_{t}(\cA)$ has maximal codimension in $R$ (this
   subset is $ W_s(\underline{b};\underline{a})$) is open. Since it is
   irreducible by Lemma~\ref{WWs}, we have proved the open irreducible property
   stated in Proposition~\ref{ghost}, and it remains to see the inclusion $
   W_s(\underline{b}_{\hat i};\underline{a}_{\hat j}) \subset
   W_s(\underline{b};\underline{a}) $ and the existence of generizations.

   We {\it claim} that the elements $R/I_{t}(\cA)$ of $
   W_s(\underline{b};\underline{a})$ whose matrix $\cA$ contains a unit of the
   field $k$ at the $(i,j)$-entry, belong to $W_s(\underline{b}_{\hat
     i};\underline{a}_{\hat j})$. 
   Indeed $\cA$ is a presentation matrix of $M = \coker \varphi^*$ and by
   rearranging the direct summands of the source and target of the morphism
   $\varphi^*$, we may assume $(i,j)=(1,0)$. By elementary row operations we
   transform $\cA$ to a matrix with only zeros (and one $1$) in the first
   column, i.e. there is an invertible $t \times t$ matrix $\cC$ such that the
   cokernel of the map induced by $\cC \cdot \cA$ is isomorphic to $M$. By
   Fitting's lemma, $I_{t}(\cA)= I_{t}(\cC \cdot \cA)$. Moreover the $(t-1)
   \times (t+c-2)$ matrix $\cA'$ obtained by deleting the first row and column
   of $\cC \cdot \cA$ satisfies $ I_{t-1}(\cA')=I_{t}(\cC \cdot \cA)$ and
   defines a determinantal ring of $W_s(\underline{b}_{\hat
     i};\underline{a}_{\hat j})$, and we have proved the inclusion and the
   whole {\it claim} (but taking any $R/I_{t}(\cA')$ of
   $W_s(\underline{b}_{\hat i};\underline{a}_{\hat j})$ and putting $\cA=
   \left(\begin{smallmatrix} 1 & 0 \\ 0 & \cA' \end{smallmatrix}\right)$ we
   get $(R/I_{t}(\cA)) \in W_s(\underline{b};\underline{a})$ which is an
   easier argument for the inclusion).

   Finally to see the existence of the generization of $R/I$ of
   $W_s(\underline{b};\underline{a}) \setminus W_s(\underline{b}_{\hat
     i};\underline{a}_{\hat j})$, $I=I_{t}(\cA)$, we may assume that $\cA$ has
   a $0$ at the $(i,j)$-entry by the proven claim. Then we apply
   Lemma~\ref{defalpha} to $T=k[u]_{\wp}$, $\wp=(u)$, 
   letting $u=0$ correspond to $\cA$ and $\cA_T$ to a matrix obtained by
   replacing the 0 of the $(i,j)$-entry by $u$ and all other $(i',j')$-entries
   by polynomials of $R \otimes k[u]$ of degree $a_{j'}-b_{i'}$ with
   coefficients in $k[u]$ such that the choice $u=1$ makes the corresponding
   matrix general enough. More precisely, with notations as above, i.e. with
   $(i,j)=(1,0)$ etc., we may choose the coefficients of the polynomials of
   the 1. column such that the choice $u=1$ make the entries equal to 0 for
   all $(i,0)$ with $i>1$ and the entries of $\cA'$ general. In some open
   subset $U \ni \{0,1\} $ of $\Spec(k[u])$ the ideal $I_{t}(\cA_u)$ of the
   matrices given by $\cA_u$ for $u \in U \setminus \{0\}$ has maximal
   codimension in $R$ and we are done.
\end{proof}

\begin{corollary} \label{corghost} Let $I= I_{t}(\cA)$, $R/I \in
  W_s(\underline{b};\underline{a}) $, and suppose that the ideal of submaximal
  minors $ I_{t-1}(\cA)=R$. Then $R/I$ is a complete intersection of $R$.
\end{corollary}
\begin{proof} Suppose $ I_{t-1}(\cA)=R$. Since the submaximal minors are
  homogeneous and generate $ I_{t-1}(\cA)$, one of the submaximal minors must
  be a unit, whence at least one of the entries of $\cA$ is a unit. Using the
  notations and the proven claim in the proof above, we have $I_{t}(\cA)=
  I_{t}(\cC \cdot \cA)= I_{t-1}(\cA')$, and we get $I_{t-1}(\cA)= I_{t-1}(\cC
  \cdot \cA)= I_{t-2}(\cA')$ for $t \ge 3$ by the same arguments. Since
  $I_{1}(\cA')$ is a c.i. of $R$, we conclude by induction on $t$.
\end{proof}
\begin{remark} \label{ghdim2} The ghost terms treated in
  Proposition~\ref{ghost} are removable under generization, while all other
  ghost terms appearing in the Eagon-Northcott resolution of $R/I_{t}(\cA)$
  may sometimes probably be removed, but not always. Indeed the latter can
  {\bf not} be removed under generization provided {\it every graded
    deformation of $R/I_{t}(\cA)$ comes from deforming $\cA$} because, in this
  case, $\overline{W_s(\underline{b};\underline{a})}$ is an irreducible
  component of $\GradAlg(H)$, whence the minimal resolution of its general
  element is given by some Eagon-Northcott resolution of $R/I_{t}(\cA')$ with
  $\cA'$ minimal, cf. Example~\ref{gh2} below. Note that
  Theorem~\ref{Amodulethm5} and Theorem~\ref{compthmvar} give conditions under
  which every graded deformation of $R/I_{t}(\cA)$ comes from deforming $\cA$.
  In particular if 
  $\dim R/I_{t}(\cA) \ge 3$ and $a_{i-\min (3,t)}\ge b_{i}$ for $\min (3,t)\le
  i \le t$, {\it only} ghost terms as in Proposition~\ref{ghost} are removable
  under generization!
\end{remark}

\begin{example} \label{gh1}
(removing all ghost terms using  Proposition~\ref{ghost})

(i) Let $R=k[x_0,x_1]$, let $\cB$ be a general $2 \times 3$ matrix with linear
(resp. quadratic) entries in the first (resp. second) row and let $\cA=
\left(\begin{smallmatrix} u & v \\ w & \cB \end{smallmatrix}\right)$ where $u
\in k$, $v$ a row of general linear forms and $w$ a column whose transpose
is $(0,x_0)$. The vanishing of all $3 \times 3$ minors defines an artinian
ring with $h$-vector $(1,2,3,1)$ for every $u \in k$. If $I$ (resp. $I_g$) is
the ideal given by all $3 \times 3$ minors of $\cA$ with $u=0$ (resp. $u=1$),
then $(R/I) \in W_s(-2,-1,-1;\;-1,0,0,0)$ and $(R/I_g) \in W_s(-2,-1;\;0,0,0)$
and using Macaulay 2 we find minimal resolutions
\begin{equation*} \label{G33} 0 \rightarrow R(-5) \oplus R(-4)^{2} \rightarrow
  R(-4) \oplus R(-3)^{3}\rightarrow I \rightarrow 0\, ,
\end{equation*}
\begin{equation*} \label{G33} 0 \longrightarrow R(-5) \oplus R(-4)
  \longrightarrow R(-3)^{3} \longrightarrow I_g \longrightarrow 0.
\end{equation*}
Here $R/I_g$ is a generization of  $R/I$ in  $\GradAlg(H)$.

(ii) Let $R=k[x_0,x_1,x_2]$, let $\cB$ be a general $2 \times 4$ matrix with
linear (resp. cubic) entries in the first (resp. second) row and let
$\cA= \left(\begin{smallmatrix} u & v \\
    w & \cB \end{smallmatrix}\right)$ where $u \in k$, $v$ a row of general
linear forms and $w$ the column whose transpose is $(0,x_0^2)$. Using
Macaulay 2 we get that the vanishing of all $3 \times 3$ minors defines an
artinian ring with $h$-vector $(1,3,6,10,9,7,3)$. If $I$ (resp. $I_g$) is the
ideal given by the $3 \times 3$ minors of $\cA$ with $u=0$ (resp. $u=1$), then
$(R/I) \in W_s(-3,-1,-1;\;-1,0,0,0,0)$ and $(R/I_g) \in W_s(-3,-1;\;0,0,0,0)$
and we have minimal resolutions
\begin{equation*} 0 \rightarrow R(-10) \oplus R(-8) \oplus R(-6) \rightarrow
  R(-7)^4 \oplus R(-5)^4 \rightarrow R(-4)^{6} \rightarrow I_g \rightarrow 0\, ,
\end{equation*}
{\small
\begin{equation*}  0 \rightarrow R(-10) \oplus R(-8)^2 \oplus R(-6)^3\rightarrow
   R(-8) \oplus R(-7)^4 \oplus R(-6)^2 \oplus R(-5)^8 \rightarrow R(-5)^4
  \oplus R(-4)^{6} \twoheadrightarrow I. 
\end{equation*} }
\end{example}

\begin{example} \label{gh2}
(removing {\it some} ghost terms using  Proposition~\ref{ghost})

 Let $R=k[x_0,x_1,x_2]$ and let
$\cA= \left(\begin{smallmatrix} u & v \\
    w & \cB \end{smallmatrix}\right)$ be as in Example~\ref{gh1} (ii) only
replacing cubic by quadratic and $w^T$ by $(0,x_0)$. Using Macaulay 2 we get
that the vanishing of all $3 \times 3$ minors defines an artinian ring with
$h$-vector $(1,3,6,4,1)$. If $I$ (resp. $I_g$) is the
ideal generated by the $3 \times 3$ minors of $\cA$ for $u=0$ (resp. $u=1$),
then $(R/I) \in W_s(-2,-1,-1;\;-1,0,0,0,0)$ and $(R/I_g) \in
W_s(-2,-1;\;0,0,0,0)$ and we have minimal resolutions
\begin{equation} \label {bres} 0 \rightarrow R(-7) \oplus R(-6) \oplus R(-5)
  \rightarrow R(-5)^4 \oplus R(-4)^4 \rightarrow R(-3)^{6} \rightarrow I_g
  \rightarrow 0\, ,
\end{equation}
{\small
  \begin{equation} \label {bres2} 0 \rightarrow R(-7) \oplus R(-6)^2 \oplus
    R(-5)^3\rightarrow R(-6) \oplus R(-5)^6 \oplus R(-4)^8 \rightarrow R(-4)^4
    \oplus R(-3)^{6} \rightarrow I \to 0.
  \end{equation} }
Observe that there is still a ghost term in the resolution of $I_g$. In the
artinian case, however, we have seen in Example~\ref{boij} that $\overline{ W_s(-2,-1;\;0,0,0,0)}$
is not an irreducible component of $\GradAlg(H)$. So there is a
possibility for removing this ghost under a generization to a non-determinantal
artinian ring! Indeed M. Boij pointed out in a note he sent me that a general
artinian algebra with Hilbert function $(1,3,6,4,1)$ can be seen as a type
two algebra $A$ given by the inverse system $(f,g)$ where $\deg f = 4$ and
$\deg g = 3$. Following Boij's note, this algebra has a minimal resolution
as in \eqref{bres} with the ghost term $R(-5)$ removed and $A$ is probably a
generization of $R/I_g$.

On the other hand, we may very well consider the determinantal algebras $R/I$
and $R/I_g$ described above in the polynomial ring $R=k[x_0,x_1,\cdots,x_n]$
for $n \ge 5$. This leads to determinantal rings of dimension greater or equal
to 3 with minimal resolutions exactly as in \eqref{bres} and \eqref{bres2}.
Thanks to Theorem~\ref{Amodulethm5}, see Remark~\ref{ghdim2},
$\overline{W_s(-2,-1;\;0,0,0,0)}$ is in this case an irreducible component of
$\GradAlg(H)$ (as well as of $\Hi^p(\PP^n)$, up to closure), whence $R/I_g$
has no generization to non-determinantal rings and the ghost term $R(-5)$ of
\eqref{bres} remains a ghost term for the general element of the component.
 \end{example}

 Finally suppose $c=2$. Then repeated use of Proposition~\ref{ghost} will
 remove $\it all$ ghost terms in the minimal resolution of $I$ since the
 transpose of $\cA$ is the Hilbert-Burch matrix. This is rather well known
 (cf. \cite[Thm.\! 2(iii)]{elli}); indeed even more precise results are true,
 see \cite{I2} which finds the codimension of all ``ghost terms strata'' when
 $\dim A$ is small and $A$ is Cohen-Macaulay. Note that any codimension-$2$
 Cohen-Macaulay quotient $R/I$ is determinantal, so Proposition~\ref{ghost}
 applies to such $R/I$. If $A=R/I$ is the homogeneous coordinate ring of a
 codimension-$2$ scheme $X$ in $\PP^n$, then $A$ is Cohen-Macaulay for $\dim
 A=1$ and there are no ghost terms in the minimal resolution of $I:=I_X$ for
 $A$ general. If $X$ is a locally Cohen-Macaulay (lCM) curve in $\proj{3}$ and
 $A$ is not necessarily Cohen-Macaulay, then we can generalize the removal of
 ghost terms above for codimension-$2$ quotients to the following.

\begin{proposition}\label{amainresRao} \cite[Thm.\! 2.8]{Krao}
  Let $X \subset \proj{3}$ be any lCM curve and let $ 0 \rightarrow L_4
  \xrightarrow{\ \sigma} L_3 \rightarrow \cdots \rightarrow M \rightarrow 0 $
  be a minimal $R$-free resolution of the Rao module $M:=H_{*}^1(\cI_{X})$. By
  \cite{R} there is a minimal $R$-free resolution of the following form
\begin{equation*} 
  0 \rightarrow L_4  \xrightarrow{ \ \sigma \oplus 0 \ } L_3 \oplus P_2
  \rightarrow P_1 \rightarrow I_X \rightarrow 0 \,  \ 
\end{equation*}
(i.e. the composition $L_4 \xrightarrow{ \ \sigma \oplus 0 \ } L_3 \oplus P_2
\to P_2$ is zero). If $P_1$ and $P_2$ have a common free summand; 
  $P_2 = P_2' \oplus R(-i)$, $P_1 = P_1' \oplus R(-i)$, 
  then there is a generization $X'$ of $X$ in $\GradAlg(H)$ with minimal
  resolution
\begin{equation*} 
  0 \rightarrow L_4  \xrightarrow{ \ \sigma \oplus 0 \ } L_3 \oplus P_2'
  \rightarrow P_1' \rightarrow I_{X'} \rightarrow 0 \, . \ 
\end{equation*}
\end{proposition}
Note that the statement ``a generization $X'$ of $X$ in $\Hi ^p(\PP^3)$ with
constant postulation'' in \cite[Thm.\! 2.8]{Krao} really means ``a
generization $X'$ of $X$ in $\GradAlg(H)$''. This result applies to a
codimension-$2$ Cohen-Macaulay quotient $R/I_X$ by letting $M=0$, which
implies $L_3=L_4=0$, coinciding with Proposition~\ref{ghost} in this case.

\section{Upgrading of previous results}

In this section we generalize main theorems from \cite{K09, KM09} by using
Theorems~\ref{Amodulethm3} and \ref{Amodulethm5} for $\dim A \ge 1$, or more
precisely we use that the two conjectures in \cite{KM09} mentioned in the
Introduction now are theorems (for
$n-c>0$). 
Indeed the previous proofs relied on the part of the conjectures which was
proved (in \cite{KM}) at the time these papers were written, but we can, using
mainly the ``same'' proofs as in \cite{K09,KM09}, get stronger results
(the same conclusions under weaker assumptions). To better understand
the proofs presented here, it may be a good idea to consider the corresponding
proofs in \cite{K09} simultaneously.

We start, however, with \cite{KM09}. The following theorem was proved in
\cite[Thm.\;3.2]{KM09} under some assumptions (mainly ``$a_0 > b_t$, $a_{t+3}
> a_{t-2}$ and $n-c \ge 0$'') and it was in \cite[Cor.\;5.7]{K10} shown to be
true in general for $n-c \ge 1$ (under elsewhere weaker assumptions: $
W(\underline{b};\underline{a}) \ne \emptyset$ and $a_{i-2} \ge b_i$ for $i \ge
2$). For $n=c$, \cite[Thm.\;3.2]{KM09} is still the best result we have with
assumptions only on $a_{j}$ and $b_{i}$, and we can immediately generalize it
to

\begin{theorem} \label{codcdim0} Let $X \subset \PP^{n}$ 
be a general determinantal scheme
  of codimension $c$ and suppose $a_0 > b_t$ and $ \ a_{t+c-2}>a_{t-2} $. Then
  (cf.\! Remark \ref{dimWcA})
  $$ \dim W(\underline{b};\underline{a})= \lambda _c+K_3+K_4+ \cdots
  +K_c \ .$$
\end{theorem}
\begin{proof} This follows from \cite[Prop.\! 3.1]{KM09}, see
  Proposition~\ref{mainN1}(i) below, and \cite[Cor.\! 5.7]{K10}.
\end{proof}

Now, to generalize the main results of \cite{K09}, we first restate
\cite[Prop.\! 3.4]{K09}. With notations as in Sec.\! 2 (i.e. we delete the
last column of $\cA$ to get $\cA_{c-1}$ and the map $\varphi_{c-1}$ induced by
the transpose of $\cA_{c-1}$), let $\cB:=\cA_{c-1}$, $B:=R/I_t(\cB)$ and
$N:=\coker \varphi_{c-1}^*$. In this section we let $A:=R/I_t(\cA)$ be good
determinantal, i.e. $(A)\in W(\underline{b};\underline{a})$ where
$\underline{a} = a_0, a_1,..., a_{t+c-2}$, and as previously we assume that
{\it all entries of the deleted column(s) belong to} $\goth m$. It follows that
$(B) \in W(\underline{b};\underline{a'})$  with $\underline{a'} = a_0,
a_1,..., a_{t+c-3}$ by \cite{Bru} and the inclusions $I_t(\cB) \subset
I_t(\cA) \subset I_{t-1}(\cB)$. We set $X=\Proj(A)$ and
$Y=\Proj(B)$.

\begin{proposition}\label{main1} \ Let $c\ge 3$, let $(X) \in
  W(\underline{b};\underline{a})$
be general and suppose 
  $ \dim W(\underline{b};\underline{a'}) \ge \lambda_{c-1}+ K_3 +
  K_4+...+K_{c-1} $ (and $\ge \lambda_{2}$ if
  $c=3$). 
 If
\begin{equation}\label{maineq}
  _0\!
  \hom_R(I_{Y},I_{X/Y}) \le  \sum _{j=0}^{t+c-3} \binom{a_j-a_{t+c-2}+n}{n} \ ,
\end{equation}
then \ \ $\dim W(\underline{b};\underline{a})= \lambda_c+ K_3 +
K_4+...+K_c.$ We also get equality in \eqref{maineq}, as well as $$\dim
W(\underline{b};\underline{a}) = \dim
W(\underline{b};\underline{a'})+\dim_k N(a_{t+c-2})_{0}-1-\ _0\!
\hom_R(I_{Y},I_{X/Y}).$$
\end{proposition}
See \cite[Prop.\! 3.4]{K09} for a proof, and notice that since $\dim A \ge 1$
and $X$ is general in $ W(\underline{b};\underline{a})$, the assumption
$\depth_{I(Z)} B \ge 2$ of \cite[Prop.\! 3.4]{K09} holds letting
$Z:=V(I_{t-1}(\cB))$.

Then \cite[Prop.\! 3.1]{KM09} is mainly (i) in the following.
\begin{proposition}\label{mainN1} 
  {\rm (i)} If $a_0 > b_t$ and $ \ a_{t+c-2}>a_{t-2} $, then \eqref{maineq}
  holds for $X$ general. \\[-3.5mm]

  {\rm (ii)} If \ $ _0\! \Ext_B^i(N,B(- a_{t+c-2})) =0$ for $i=1$ and $2$,
  then \eqref{maineq} holds. In particular \eqref{maineq} holds provided
  $\depth_{I_{t-1}(\cB)}B \ge 3$ (and $\Ext_B^1(N,B) =0$ provided 
  $\depth_{I_{t-1}(\cB)}B \ge 2$).
\end{proposition}

\begin{proof}(ii) The arguments we apply in the proof are very similar to those
  needed in \eqref{specseq}, \eqref{specseq1} and \eqref{homext} to get
  related statements, except that we now use the spectral sequence
  $E_2^{p,q}:= \Ext_B^p(\Tor_q^R(B,N),B) \ \Rightarrow \ \Ext_R^{p+q}(N,B)$
  and its corresponding 5-term exact
  sequence. 
  Indeed due to assumptions and \eqref{Di} we get isomorphisms
$$ _v\!
\Ext_R^1(N,B) \simeq \ (E_2^{0,1})_v \simeq \ _v\!\Hom_R(I_Y,\Hom_B(N,B))
\simeq \ _v\!\Hom_R(I_Y,I_{X/Y}(a_{t+c-2})) \, $$ for $v= -a_{t+c-2}$,.
Moreover as in \eqref{homext}, cf. \eqref{gradedmorfismo}, we get for this $v$
the exact sequence
\begin{equation} \label{homextN}
  0 \to  \ _v\! \Hom_R(N,B) \to\ _v\! \Hom_R(F^*,B) \to \ _v\!
  \Hom_R(G_{c-1}^*,B) \to \ _v\! \Ext_R^1(N,B) \to 0
\end{equation}
where $ _v\! \Hom_R(F^*,B) \simeq \oplus_{i=1}^{t}B_{(b_i+v)}=0$ and $ _v\!
\hom_R(G_{c-1}^*,B) \simeq \dim (\oplus_{j=0}^{t+c-3}B_{(a_j+v)}) = \sum
_{j=0}^{t+c-3} \binom{a_j-a_{t+c-2}+n}{n}$. This implies \eqref{maineq}.

The above arguments are used in \cite[Rem.\! 3.4]{KM15} and the proof of
\cite[Prop.\! 3.5]{KM15} from which we get \eqref{maineq} under the assumption
$\depth_{I_{t-1}(\cB)}B \ge 4$. To see that we can weaken the depth assumption
to $\depth_{I_{t-1}(\cB)}B \ge 3$, we really need to refine the argument.
Following, however, the proof of \cite[Thm.\! 4.5]{K10} exactly as in the
paragraph after (4.3) in \cite{K10}, we get $ \Ext_B^i(N,B) =0$ for $i=1$ and
$2$ (resp. for $i=1$) under the assumption $\depth_{I_{t-1}(\cB)}B \ge 3$
(resp. $\depth_{I_{t-1}(\cB)}B \ge 2$), and we are done.
\end{proof}

\begin{remark} Using Propositions~\ref{main1} and \ref{mainN1}(ii) one may
  reprove the dimension formula
  $ \dim W(\underline{b};\underline{a})= \lambda _c+K_3+K_4+ \cdots +K_c \ $
  of Theorem~\ref{Amodulethm3} in the case $n-c \ge 1$, $a_{i-2} \ge b_i$ for
  $i\ge 2$ and $W(\underline{b};\underline{a}) \ne \emptyset$ by using the
  recursive strategy of successively deleting columns of $\cA$ from the
  right-hand side, see Introduction and Remark \ref{dimWcA}. In
  \cite{KM15} we pointed out that \cite[Cor.\! 3.19]{KM15} shows that the
  recursive strategy also applies to reprove the generic smoothness of
  $ W(\underline{b};\underline{a})$ of Theorem~\ref{Amodulethm5} in the case
  $n-c \ge 2$ and $a_{0} \ge b_t$ (due to Remark~\ref{dep}, the latter
  assumption may here be weakened to $a_{i-\min (3,t)}\ge b_{i}$ for
  $\min (3,t)\le i \le t$). This means that we have two rather different
  proofs for the two conjectures of \cite{KM09}.
\end{remark}
With the above propositions we start by considering determinantal curves. In
this and later results we denote by $\tau_{X/Y}$
the following morphism induced by  $I_{X/Y}
\hookrightarrow B$: 
\begin{equation*} 
 \tau_{X/Y} \ : \ _0\!
  \Ext^1_{B}(I_{Y}/I^2_{Y},I_{X/Y}) \rightarrow \ _0\! \Ext
  ^1_{B}(I_{Y}/I^2_{Y},B).
\end{equation*}
A key result of \cite{K09} (Prop.\! 4.6) shows that if $\ker \tau_{X/Y}=0$,
$\depth_{I_{t-1}(\cB)}B \ge 3$ and $X$ good determinantal, then the property
``every deformation of a determinantal scheme comes from deforming the
matrix'' is transferred from $Y$ to $X$. Recalling that
$\dim W(\underline{b};\underline{a})= \lambda_c + K_3+...+K_c$ is now a
theorem when $n-c=1$, 
the following result generalizes \cite[Prop.\! 4.15]{K09} to arbitrary
$c \ge 3$. (For $c=2$ we have a more complete result with stronger conclusions
in Theorem~\ref{elling}).
%
\begin{proposition} \label{Dim1Prop} Let $X=\Proj(A)\subset \PP^n$,
  $A=R/I_t(\cA)$, be general in $ W(\underline{b};\underline{a})$ and suppose
  $a_{i-\min (3,t)}\ge b_{i}$ for $\min (3,t)\le i \le t$, $c \ge 3$ and
  $\dim X =n-c=1$. If \ $Y=\Proj(B)$ is defined by the vanishing of the
  maximal minors of the matrix obtained by deleting the last column 
  of ${\cA}$, then the following statements hold:

  \vskip 2mm {\rm (i)} If $\tau_{X/Y}:\, _0\!
  \Ext^1_{B}(I_{Y}/I^2_{Y},I_{X/Y}) \rightarrow \! _0\! \Ext
  ^1_{B}(I_{Y}/I^2_{Y},B)$ is injective, then $X$ is unobstructed and
  $\overline{ W(\underline{b};\underline{a})}$ is a generically smooth
  irreducible component of \ $\Hi ^p(\PP^{n})$. Moreover  
  every deformation of $X$ comes from deforming ${\cA}$.

  {\rm (ii)} If $_0\! \Ext^1_{A}(I_{X}/I^2_{X},A)=0$, then $X$ is
  unobstructed,  and $$\codim_{\Hi ^p(\PP^{n})}\overline{
    W(\underline{b};\underline{a})}=\dim \ker \tau_{X/Y} - \ _0\! \ext
  ^1_{B}(I_{X/Y},A) \;.$$

  {\rm (iii)} We always have 
\begin{equation} \label{coiii}
\codim_{\Hi ^p(\PP^{n})} \overline{
    W(\underline{b};\underline{a})} \le \dim \ker \tau_{X/Y}. 
\end{equation}
Moreover if $\ _0\! \Ext ^1_{B}(I_{X/Y},A)=0$, then we have equality in
\eqref{coiii} if and only if $X$ is unobstructed.
 \end{proposition}
 \begin{proof} Thanks to Proposition~\ref{mainN1}(ii) the proof is the
   ``same'' as for \cite[Prop.\! 4.15]{K09}. Indeed for (i) we apply
   Theorem~\ref{Amodulethm5} onto $ W(\underline{b};\underline{a'}) \ni (Y)$
   instead of the corresponding result of \cite{K09} which required $c-1 \le
   4$ and we get (i) from \cite[Thm.\! 4.6]{K09} and Lemma~\ref{unobst}. For
   (ii) and (iii) we need the final dimension formula of
   Proposition~\ref{main1} to use the proof of \cite{K09}. Using
   Proposition~\ref{mainN1}(ii) we get the mentioned dimension formula and we
   are done.
 \end{proof}
 We consider now zero dimensional determinantal schemes ($n-c=0$). In this
 case we also need to consider the morphism 
 \begin{equation} \label{ro} \rho^1:=\rho^1_{X/Y}: \
   _0\!\Ext^1_B(I_{X/Y},I_{X/Y})\to \ _0\!\Ext^1_B(I_{X/Y},B)
\end{equation}
induced by $I_{X/Y} \hookrightarrow B$. In addition to \cite[Prop.\! 4.6]{K09}
which states that $\ker \tau_{X/Y}=0$, $\ker \rho^1=0$ and
$\depth_{I_{t-1}(\cB)}B = 2 $ transfer the property ``every deformation of a
determinantal ring comes from deforming the matrix'' from $B$ (or $Y$) to $A$,
we need a variation of \cite[Prop.\! 4.6]{K09} which don't require
$\ker \rho^1=0$ (see \cite[Prop.\! 4.13]{K09} for the details). Below we
generalize \cite[Thm.\! 4.19]{K09} by replacing the assumption $4 \le c \le 6$
in \cite[Thm.\! 4.19]{K09} by $c \ge 4$. Note that we denote $\GradAlg(H)$ by
$ \Hi ^H(\PP^n)$ since $\dim A = 1$, cf. the text before \eqref{Grad} and see
Remark \ref{dimWcA} for computing $\dim W(\underline{b};\underline{a})$. We refer to \cite{K09} for examples. 

\begin{theorem} \label{Dim0Thm} Let $X=\Proj(A) \subset \PP^n $,
  $A=R/I_t(\cA)$, be general in $ W(\underline{b};\underline{a})$ and let
  $Y=\Proj(B)$ and $ V=\Proj(C)$ be defined by the vanishing of the maximal
  minors of ${\cB}$ and ${\cC}$ respectively where ${\cB}$ (resp. ${\cC}$) is
  obtained by deleting the last column of ${\cA}$ (resp. ${\cB}$). Suppose
  $\dim X =n-c=0$, $a_{i - \min (3,t)}\ge b_{i}$ for $\min (3,t)\le i \le t$
  and suppose that \eqref{maineq} (cf.\! Proposition~\ref{mainN1}) holds.
  Moreover suppose
  $$ {\rm \ either \ \ } c=3 {\rm \ \ or \ \ [ } \ c \ge 4 \ \ {\rm and \ }
  \ker \tau_{Y/V} = 0 {\rm \ ] } . $$
  Then \ $\dim W(\underline{b};\underline{a})= \lambda_c + K_3+...+K_c$ \ and
  the following statements are true:

  \vskip 2mm {\rm (i)} If both
  $\rho^1: \, _0\!\Ext^1_B(I_{X/Y},I_{X/Y})\to \, _0\!\Ext^1_B(I_{X/Y},B)$ and
  $\tau_{X/Y}:\, _0\! \Ext^1_{B}(I_{Y}/I^2_{Y},I_{X/Y}) \rightarrow \,_0\!
  \Ext ^1_{B}(I_{Y}/I^2_{Y},B)$
  are injective, then $A$ is unobstructed and
  $\overline{ W(\underline{b};\underline{a})}$ is a generically smooth
  irreducible component of \! $\Hi^H(\PP^n)$. Moreover every deformation of
  $A$ comes from deforming ${\cA}$.

   {\rm (ii)} If $\ _0\! \Ext ^1_{B}(I_{X/Y},A)=0$ and \ $ \ker \tau_{X/Y}=
   0$, then $\overline{ W(\underline{b};\underline{a})}$ belongs to a unique
   generically smooth irreducible component $Q$ of $\Hi^H(\PP^n)$ and the
   codimension of $\overline{ W(\underline{b};\underline{a})}$ in
   $\Hi^H(\PP^n)$ is $\dim \ker \rho^1$. Indeed $A$ is unobstructed and $$\dim
   Q =\lambda_c + K_3+...+K_c + \dim \ker \rho^1 . $$

   {\rm (iii)} If \ $_0\! \Ext^1_{A}(I_{X}/I^2_{X},A)=0$, then $A$ is
   unobstructed and $$\codim_{\Hi^H(\PP^n)}\overline{
     W(\underline{b};\underline{a})}=\dim \ker \rho^1 + \dim \ker \tau_{X/Y}-
   \ _0\! \ext ^1_{B}(I_{X/Y},A) .$$

   {\rm (iv)} We always have \ $\codim_{\Hi^H(\PP^n)}\overline{
     W(\underline{b};\underline{a})} \le \dim \ker \rho^1 + \dim \ker
   \tau_{X/Y}$. \\
   Suppose $\ _0\! \Ext ^1_{B}(I_{X/Y},A)=0$. Then we
   have $$\codim_{\Hi^H(\PP^n)}\overline{ W(\underline{b};\underline{a})} =
   \dim \ker \rho^1+ \dim \ker \tau_{X/Y} \ $$ if and only if $A$ is
   unobstructed.
 \end{theorem}
 \begin{proof} Using Proposition~\ref{main1} we get
   $\dim { W(\underline{b};\underline{a})}= \lambda_c + K_3+...+K_c$ because
   Theorem~\ref{Amodulethm3} applies to
   $(B) \in W(\underline{b};\underline{a'}) $ to find
   $\dim { W(\underline{b};\underline{a'})}$ where
   $\underline{a'} = a_0, a_1,..., a_{t+c-3}$. Then, due to \cite[Prop.\! 4.6
   and Prop.\! 4.13]{K09} the proof of the theorem
   is ``the same'' as the proof of \cite[Thm.\! 4.19]{K09} since we already
   have generalized \cite[Prop.\! 4.15]{K09} to Proposition~\ref{Dim1Prop} of
   this paper to cover determinantal curves $Y$ of arbitrary codimension
   $c\ge 3$. Note also that we have to delete a column of $\cB$ (thus defining
   $V$) when we apply Proposition~\ref{Dim1Prop} to $Y$.
 \end{proof}
 
 \begin{remark} Looking at the proofs we see that we don't need to suppose
   \eqref{maineq} to get (the part of the) conclusions of (i) and (ii) that
   don't involve dimension and codimension formulas. Moreover note the overlap
   in (ii) and (iv) of the theorem.
\end{remark}

Using again the variation \cite[Prop.\! 4.13]{K09} we generalize
\cite[Prop.\! 4.24]{K09} to the following.

\begin{proposition} \label{Dim0Prop} With notations as in the first sentence
  of Theorem~\ref{Dim0Thm}, suppose $c \ge 4$, $\dim X =n-c=0$ and $X$
  general, $a_{i -\min (3,t)}\ge b_{i}$ for $\min (3,t)\le i \le t$ and
  suppose that \eqref{maineq} holds. Then
  $\dim { W(\underline{b};\underline{a})}= \lambda_c + K_3+...+K_c$ (cf.
  Remark \ref{dimWcA}), and the following statements are true:

  {\rm (i)} If $\ _0\! \Ext ^1_{B}(I_{X/Y},A)=0$, $_0\!
  \Ext^1_{B}(I_{Y}/I^2_{Y},I_{X/Y})=0$ and \ $_0\!
  \Ext^1_{B}(I_{Y}/I^2_{Y},B)=0$ then $A$ is unobstructed. Moreover $
  W(\underline{b};\underline{a})$ is contained in a unique generically smooth
  irreducible component of $\Hi ^H(\PP^{c})$ of codimension $\dim \ker
  \rho_{X/Y}^1 + \dim \ker \tau_{Y/V}- \ _0\! \ext ^1_{C}(I_{Y/V},B)$.

   {\rm (ii)} We always have $\codim_{\Hi ^H(\PP^{c})}\overline{
     W(\underline{b};\underline{a})} \le \dim \ker \rho_{X/Y}^1 + \dim \ker
   \tau_{X/Y}+\dim \ker \tau_{Y/V}$.  \\
   Suppose $\ _0\! \Ext ^1_{B}(I_{X/Y},A)=0$ and $\ _0\! \Ext
   ^1_{C}(I_{Y/V},B)=0$. Then we have $$\codim_{\Hi ^H(\PP^{c})}\overline{
     W(\underline{b};\underline{a})} = \dim \ker \rho_{X/Y}^1+ \dim \ker
   \tau_{X/Y}+\dim \ker \tau_{Y/V}$$ if and only if $A$ is unobstructed (e.g.
   $_0\! \Ext^1_{A}(I_{X}/I^2_{X},A)=0$).
 \end{proposition}
 \begin{proof} Again we have $\dim { W(\underline{b};\underline{a})}=
   \lambda_c + K_3+...+K_c$ by Proposition~\ref{main1}. Since
   Proposition~\ref{Dim1Prop} applies for {\it every} $c \ge 3$ we can use the
   proof of \cite[Prop.\! 4.24]{K09} to get Proposition~\ref{Dim0Prop}.
   Indeed we use Proposition~\ref{Dim1Prop}(ii) and \cite[Prop.\! 4.13]{K09}
   to get (i) above while we use the proof of
   Proposition~\ref{Dim1Prop}(iii) to get (ii) (see \cite{K09} for details).
 \end{proof}

{\bf Acknowledgement.} I thank prof. R.M. Mir\'o-Roig at Barcelona for interesting discussions on
this topic, prof. A. Iarrobino at Northeastern University, Boston and prof. M.
Boij at KTH, Stockholm for comments and helping me approach the artinian case.
Indeed this study of families of artinian determinantal rings started some time
ago during a visit by A. Iarrobino and M. Boij at Oslo University College
where we concretely studied examples of such families. I also thank the
referee for some valuable comments.


\end{document}